\documentclass{amsart}

\usepackage{amssymb,amsxtra,latexsym}
\usepackage{mathrsfs}
\usepackage{graphicx}
\usepackage[pagebackref]{hyperref} 
\usepackage{textcomp}
\usepackage[scr=rsfs]{mathalfa}
\usepackage{epstopdf}
\usepackage{tikz}
\usepackage{pgfplots}
\usepackage{tikz-cd}
\tikzcdset{arrow style=math font}
\tikzcdset{diagrams={nodes={inner sep=3pt}}}
\usepackage[active]{srcltx} % SRC Specials

\numberwithin{equation}{section}

\swapnumbers

\newtheorem{theorem}[equation]{Theorem}

\newtheorem{lemma}[equation]{Lemma} 
\newtheorem{proposition}[equation]{Proposition}
\newtheorem{problem}[equation]{Problem}

\theoremstyle{definition}
\newtheorem{definition}[equation]{Definition} 

\theoremstyle{remark}
\newtheorem{remark}[equation]{Remark} 
\newtheorem{remarks}[equation]{Remarks} 
\newtheorem{example}[equation]{Example}

\newcounter{num}

\newenvironment{numerate}%
{\begin{list}{\hskip\labelsep\hskip\parindent(\roman{num})}%
{\usecounter{num}%
\setlength\leftmargin{0em}\setlength\topsep{0em}%
\setlength\parsep{0em}\setlength\partopsep{0em}%
\setlength\itemsep{0em}\setlength\labelwidth{0em}}}%
{\end{list}}

\definecolor{ceruleanblue}{rgb}{0.16, 0.32, 0.75}
\hypersetup{colorlinks=true,allcolors=ceruleanblue}
\newcommand\group[1]{{\text{\bf#1}}}

\newcommand\SO{\group{SO}}
\newcommand\GL{\group{GL}}

\newcommand\U{\group{U}}
\newcommand\Sp{\group{Sp}}

\newcommand\bb{\mathbb}

\newcommand\Z{\bb{Z}} 
\newcommand\Q{\bb{Q}}
\newcommand\R{\bb{R}} 
\newcommand\C{\bb{C}}

\newcommand\ca{\mathcal}

\newcommand\F{\ca{F}}
\renewcommand\H{\ca{H}}

\newcommand\lie{\mathfrak}

\newcommand\X{\lie{X}}

\newcommand\abs[1]{\lvert#1\rvert}

\newcommand\inner[1]{\langle#1\rangle}

\newcommand\qu[1][\kern.3ex]{/\kern-.7ex/_{\kern-.4ex#1}}
\newcommand\bigqu[1][\,\,]{\big/\kern-.85ex\big/_{\!\!#1}}

\DeclareMathOperator\Ad{Ad}

\DeclareMathOperator\diag{diag}
\DeclareMathOperator\Lie{Lie}
\DeclareMathOperator\per{per}
\DeclareMathOperator\Per{Per}
\DeclareMathOperator\pr{pr}
\DeclareMathOperator\rank{rank}
\DeclareMathOperator\ray{ray}
\DeclareMathOperator\tr{tr}

\newcommand\longto{\longrightarrow} 
\newcommand\bas{{\mathrm{bas}}}

\renewcommand\Re{\operatorname{Re}}
\renewcommand\Im{\operatorname{Im}}

\newcommand\note[1]{$^\dag$\marginpar{\footnotesize{$^\dag${#1}}}}

\newcommand\hM{M\kern-0.7em\widehat{\phantom I}\kern0.3em}

\begin{document}

\title[Conformal symplectic convexity]{The convexity package for
  Hamiltonian actions on conformal symplectic manifolds}

%---------------------------------------
\author{Youming Chen}

\address{Y. Chen\\School of Science\\Chongqing University of
  Technology\\Chongqing 400054, China}

\curraddr{Department of Mathematics\\Pennsylvania State
  University\\USA}

\email{youmingchen@cqut.edu.cn}

%================================
\author{Reyer Sjamaar}

\address{R. Sjamaar\\Department of Mathematics\\Cornell
  University\\Ithaca, NY 14853-4201, USA}

\email{sjamaar@math.cornell.edu}

%=================================
\author{Xiangdong Yang}

\address{X. Yang\\Department of Mathematics\\Chongqing
  University\\Chongqing 401331, China}

\curraddr{Department of Mathematics\\Cornell University\\Ithaca, NY
  14853-4201, USA}

\email{xy373@cornell.edu}

%--------------------------------------

\thanks{This work was partly supported by the National Natural Science
  Foundation of China (grants No.\ 11571242 and No.\ 11701051) and the
  China Scholarship Council.}

\subjclass[2010]{53D20; 58D19}

\keywords{Conformal symplectic manifolds; moment maps; convexity
  properties.}

\date{\today}

%\dedicatory{ }

%\commby{}

%%%%%%%%%%%%%%%%%%%%%%%%%%%%%%%%%%%%%%%%%%%%%%%%%%%%%%%%%%%%%%%%%%%%%%%%

\begin{abstract}
Consider a Hamiltonian action of a compact connected Lie group on a
conformal symplectic manifold.  We prove a convexity theorem for the
moment map under the assumption that the action is of Lee type, which
establishes an analog of Kirwan's convexity theorem in conformal
symplectic geometry.
\end{abstract}

%%%%%%%%%%%%%%%%%%%%%%%%%%%%%%%%%%%%%%%%%%%%%%%%%%%%%%%%%%%%%%%%%%%%%%%%

\maketitle

\tableofcontents

%%%%%%%%%%%%%%%%%%%%%%%%%%%%%%%%%%%%%%%%%%%%%%%%%%%%%%%%%%%%%%%%%%%%%%%%
\section{Introduction}
%%%%%%%%%%%%%%%%%%%%%%%%%%%%%%%%%%%%%%%%%%%%%%%%%%%%%%%%%%%%%%%%%%%%%%%%

One of the most interesting results concerning the geometry of
Hamiltonian actions of compact Lie groups on symplectic manifolds is
the convexity theorem for the moment map image.  An early instance of
this theorem is Kostant's result~\cite{Ko73} of 1973 for torus actions
on conjugacy classes and flag manifolds.  In 1982, Atiyah~\cite{At82}
and Guillemin-Sternberg~\cite{GS82} independently proved the convexity
theorem in the case of torus actions.  For an action of a non-abelian
Lie group the moment map image is in general not convex.  However,
Guillemin-Sternberg~\cite{GS82} showed that if $M$ is an integral
K\"{a}hler manifold the quotient space of the moment image by the
coadjoint $K$-action is convex and they conjectured the same result
for general symplectic manifolds.  Subsequently, Kirwan \cite{Ki84}
confirmed Guillemin and Sternberg's conjecture in 1984.  The
symplectic convexity theorem has versions for many other geometric
objects, such as symplectic orbifolds, contact manifolds, generalized
complex manifolds, presymplectic manifolds, and so on (see
e.g.\ \cite{LT97,Le02,Ni09,CK10,RZ17,LS17,GMPS17}).

From a conformal point of view, the closest geometric structure to a
symplectic structure is a \emph{(locally) conformal(ly) symplectic}
structure.  This type of structure was first introduced by
Lee~\cite{Le43} and studied by Libermann~\cite{Li55},
Lefebvre~\cite{Le69}, Gray-Hervella~\cite{GH80}, Vaisman \cite{Va85},
Eliashberg-Murphy~\cite{EM15}, and Chantraine-Murphy~\cite{CM16}.
(The term ``locally conformally symplectic'' is current in much of the
literature, but for the sake of brevity we will follow \cite{EM15} and
say ``conformal symplectic''.)  Conformal symplectic manifolds serve
as natural phase spaces for certain problems in mechanics, such as
Gaussian isokinetic dynamics and Nos\'{e}-Hoovers dynamics
(cf.\ \cite{WL98}).  Conformal symplectic geometry relates closely to
several other geometries.  For instance, it is known that many
non-K\"{a}hler compact complex surfaces admit locally conformal
K\"ahler forms; see~\cite{Be00,OVV18} and the references therein.
Conformal symplectic manifolds are nothing but even-dimensional
transitive Jacobi manifolds; see~\cite{GL84}.  A conformal symplectic
manifold of the first kind is a special example of a contact pair and
therefore possesses an intrinsic transversely symplectic foliation;
see~\cite{BK11}.  Thus the geometry of conformal symplectic manifolds
of the first kind is a special type of transversely symplectic
geometry.

Belgun-Goertsches-Petrecca~\cite{BGP18} have recently extended the
convexity theorem of Atiyah and Guillemin-Sternberg to Hamiltonian
torus actions on conformal symplectic manifolds.  The purpose of this
paper is to generalize their result to the nonabelian case.

\begin{theorem}\label{theorem;main}
Let $K$ be a compact connected Lie group which acts on a connected
strict conformal symplectic manifold $(M,\omega,\theta)$ in a
Hamiltonian fashion.  Assume that the $K$-action is of Lee type and
that the moment map $\Phi\colon M\to\lie{k}^*$ is proper, where
$\lie{k}=\Lie(K)$.  Choose a maximal torus $T$ of $K$ and a closed
Weyl chamber $C$ in $\lie{t}^*$, where $\lie{t}=\Lie(T)$.
\begin{enumerate}
\item
The fibres of $\Phi$ are connected and $\Phi\colon M\to\Phi(M)$ is an
open map.
\item
$\Delta(M)=\Phi(M)\cap C$ is a closed convex polyhedral set.
\end{enumerate}
\end{theorem}

%Here is an outline of the paper.  In Section~\ref{sec-2} we introduce
%notation and review the definition of a locally conformal symplectic
%structure.  In Section~\ref{sec-3}, we introduce Hamiltonian actions
%on locally conformal symplectic manifolds and establish a local
%model. 

The proof this theorem, which is a variation on the argument
of~\cite{BGP18}, is contained in Section~\ref{section;convex}.  We
also state a version for global conformal symplectic manifolds,
Theorem~\ref{theorem;global}, and a version for conformal
presymplectic manifolds,
Theorem~\ref{theorem;conformal-presymplectic}, in
Appendix~\ref{section;conformal-presymplectic}.  In
Section~\ref{section;conformal-pre} we investigate the relation
between the moment map image of $M$ and that of its symplectic
covering manifolds $\tilde{M}$, and obtain
Theorems~\ref{theorem;rank-1} and~\ref{theorem;rank-2}, which also
generalize results for torus actions obtained in~\cite{BGP18}.  In
addition we introduce a foliation of $M$ into presymplectic manifolds,
which we call the \emph{Lee foliation}, and we prove a stability
result, Theorem~\ref{theorem;lee-leaf}, which combines a form of Reeb
stability with a stability statement for the moment map images of the
leaves of the Lee foliation.  In Section~\ref{section;kind}, we
present an alternative proof of the convexity theorem for conformal
symplectic manifolds of the first kind from the viewpoint of
transverse symplectic geometry
(Theorem~\ref{theorem;conformal-first}).  The properness hypothesis on
the moment map in Theorem~\ref{theorem;main} must not be taken too
seriously.  It is violated in some cases of interest where the
conclusion of the theorem nevertheless holds.  A few such examples are
discussed in Section~\ref{section;examples} and
Appendix~\ref{section;real}.

%%%%%%%%%%%%%%%%%%%%%%%%%%%%%%%%%%%%%%%%%%%%%%%%%%%%%%%%%%%%%%%%%%%%%%%%
\subsection*{Acknowledgements}
%%%%%%%%%%%%%%%%%%%%%%%%%%%%%%%%%%%%%%%%%%%%%%%%%%%%%%%%%%%%%%%%%%%%%%%%

Youming Chen and Xiangdong Yang would like to thank the Departments of
Mathematics at Pennsylvania State University and Cornell University,
respectively, for their hospitality and their excellent working
environment.  They would like to express their great gratitude to the
China Scholarship Council for financially supporting their visits.
Reyer Sjamaar and Xiangdong Yang are grateful to Yi Lin for many
fruitful discussions.

%%%%%%%%%%%%%%%%%%%%%%%%%%%%%%%%%%%%%%%%%%%%%%%%%%%%%%%%%%%%%%%%%%%%%%%%
\section{Preliminaries}\label{section;preliminary}
%%%%%%%%%%%%%%%%%%%%%%%%%%%%%%%%%%%%%%%%%%%%%%%%%%%%%%%%%%%%%%%%%%%%%%%%

In this section we establish some notational conventions and present a
brief review of conformal symplectic structures.

%%%%%%%%%%%%%%%%%%%%%%%%%%%%%%%%%%%%%%%%%%%%%%%%%%%%%%%%%%%%%%%%%%%%%%%%
\subsection*{Notation and conventions}
%%%%%%%%%%%%%%%%%%%%%%%%%%%%%%%%%%%%%%%%%%%%%%%%%%%%%%%%%%%%%%%%%%%%%%%%

Throughout this paper $K$ will denote a compact connected Lie group
with Lie algebra $\lie{k}$ and $M$ will denote a paracompact smooth
manifold on which $K$ acts smoothly.  We denote by $\X(M)$ the Lie
algebra of smooth vector fields on $M$.  The stabilizer of a point
$m\in M$ is denoted by $K_m$ and the orbit through $m$ by $K\cdot m$.
The vector field on $M$ generated by the infinitesimal action of
$\xi\in\lie{k}$ is denoted by $\xi_M$.  A differential form
$\alpha\in\Omega^{p}(M)$ is \emph{horizontal} if
$\iota(\xi_M)\alpha=0$ for all $\xi\in\lie{k}$.  A $K$-invariant
horizontal form $\alpha$ is said to be \emph{basic} with respect to
the $K$-action.  We denote by $\Omega^{p}_{K,\bas}(M)$ the space of
basic $p$-forms.  Since the exterior differential operator $d$
preserves the basic forms we get a complex $(\Omega^*_{K,\bas}(M),d)$.
The associated cohomology, denoted by $H^*_{K,\bas}(M)$, is called the
\emph{basic cohomology} of the $K$-manifold~$M$.

We denote by $\R_{>0}$ the multiplicative group of positive real
numbers and by $\ca{C}^\infty(M,\R_{>0})$ the group of smooth
positive-valued functions on $M$.

Let $V$ be a finite-dimensional vector space over $\R$.  A
\emph{closed affine halfspace} is a subset of $V$ defined by an
inequality $\inner{\phi,v}\ge c$ with $\phi\in V^*$ and $c\in\R$.  A
\emph{convex polyhedral set} is the intersection of a locally finite
collection of closed affine halfspaces in $V$.  A \emph{convex
  polyhedron} is the intersection of finitely many closed affine
halfspaces.  A \emph{convex polytope} is a bounded convex polyhedron.
Suppose $V$ is defined over $\Q$.  A closed affine halfspace is
\emph{rational} (resp.\ \emph{semirational}) if it is defined by an
inequality $\inner{\phi,v}\ge c$ with $\phi\in V_\Q^*$ and $c\in\Q$
(resp.\ $\phi\in V_\Q^*$ and $c\in\R$).  A \emph{rational}
(resp.\ \emph{semirational}) convex polyhedral set is the intersection
of a locally finite collection of rational (resp.\ semirational)
closed affine halfspaces in~$V$.  Thus a semirational convex
polyhedral set has a rational normal vector at each facet, but its
vertices may not be rational.

%%%%%%%%%%%%%%%%%%%%%%%%%%%%%%%%%%%%%%%%%%%%%%%%%%%%%%%%%%%%%%%%%%%%%%%%
\subsection*{Conformal symplectic structures}
%%%%%%%%%%%%%%%%%%%%%%%%%%%%%%%%%%%%%%%%%%%%%%%%%%%%%%%%%%%%%%%%%%%%%%%%

Let $M$ be a smooth manifold.  An \emph{almost symplectic form} is a
non-degenerate $2$-form on $M$.  An almost symplectic form $\omega$ is
\emph{conformal symplectic} if there exist an open covering
$\{U_\lambda\}_{\lambda\in\Lambda}$ of $M$ and a family of smooth
functions $\{f_\lambda\colon U_\lambda\to\R\}$ such that
$\exp(f_\lambda)\cdot(\omega|_{U_\lambda})$ is symplectic.  The
non-degeneracy of $\omega$ implies $df_\lambda=df_{\mu}$ on
$U_\lambda\cap U_{\mu}$.  Consequently, the forms $\{df_\lambda\}$
glue to a globally defined closed 1-form $\theta$, which satisfies
$d\omega=-\theta\wedge\omega$.  If $\dim M\ge4$, the form $\theta$ is
uniquely determined by $\omega$.

\begin{definition}
A \emph{conformal symplectic structure} on $M$ is a pair
$(\omega,\theta)$, where $\omega$ is an almost symplectic form and
$\theta$ a closed $1$-form with the property that
$d\omega+\theta\wedge\omega=0$.  The form $\theta$ is the \emph{Lee
  form} and its cohomology class $[\theta]\in H^1(M;\R)$ is the
\emph{Lee class} of the conformal symplectic structure.  A conformal
symplectic structure $(\omega,\theta)$ is \emph{strict} if $\theta$ is
not exact, and \emph{global} if $\theta$ is exact.  A \emph{conformal
  symplectic manifold} is a triple $(M,\omega,\theta)$, where
$(\omega,\theta)$ is a conformal symplectic structure on $M$.
\end{definition}

The Lee form associated with a conformal symplectic structure defines
a flat connection on the trivial line bundle $M\times\R$, or
equivalently on the trivial principal bundle $M\times\R_{>0}$.  The
group of gauge transformations of this principal bundle is the group
$\ca{C}^\infty(M,\R_{>0})$ of smooth positive-valued functions.  The
gauge action of a function $a>0$ on a conformal symplectic structure
$(\omega,\theta)$ is given by
\[a\cdot(\omega,\theta)=(a\omega,\theta-d\log a).\]
Conformal symplectic structures related by the gauge action are called
\emph{conformally equivalent}.  A conformal symplectic structure
$(\omega,\theta)$ is global if and only if it is conformally
equivalent to a symplectic structure, namely
$e^f\cdot(\omega,\theta)=(e^f\omega,0)$ if $\theta=df$.

It is known that every compact almost symplectic manifold
$(M,\omega_0)$ with nonzero cohomology group $H^1(M;\Z)$ admits a
conformal symplectic structure.  More precisely, the $h$-principle for
conformal symplectic manifolds of Eliashberg-Murphy~\cite[Theorem
  1.8]{EM15} states that for each nonzero class $\Theta\in H^1(M;\Z)$
and for all sufficiently large $c>0$ there exists a conformal
symplectic form $\omega$ which is homotopic to $\omega_0$ through
almost symplectic forms and has Lee class equal to
$[\theta]=c\,\Theta$.

Let $(M,\omega,\theta)$ be a conformal symplectic manifold.  The
connection $\theta$ defines a covariant derivative, also called the
\emph{twisted differential},
$d_\theta\colon\Omega^*(M)\to\Omega^{*+1}(M)$ given by
\[d_\theta\alpha=d\alpha+\theta\wedge\alpha.\]
By definition $\omega$ and $\theta$ are $d_\theta$-closed.  For a
vector field $X\in\X(M)$ the \emph{twisted Lie derivative}
$L_\theta(X)\colon\Omega^*(M)\to\Omega^*(M)$ is the operator given by
\[L_\theta(X)\alpha=L(X)\alpha+\theta(X)\alpha.\]
These operators satisfy the usual rules of the Cartan differential
calculus, namely
\begin{equation}\label{equation;cartan}
\begin{aligned}
{[\iota(X),\iota(Y)]}&=0,&\quad
[L_\theta(X),L_\theta(Y)]&=L_\theta([X,Y]),\\
[L_\theta(X),d_\theta]&=0,&\quad[L_\theta(X),\iota(Y)]&=\iota([X,Y]),\\
[d_\theta,d_\theta]&=0,&\quad[\iota(X),d_\theta]&=L_\theta(X)
\end{aligned}
\end{equation}
for all $X$, $Y\in\X(M)$, where the square brackets denote graded
commutators.  In particular, $(\Omega^*(M),d_\theta)$ is a cochain
complex, the cohomology of which we denote by $H_\theta^*(M)$.  The
Lie algebra of vector fields $\X(M)$ has several subspaces of
interest, including
\begin{align*}
\X_\omega(M)&=\{\,X\in\X(M)\mid L(X)\omega=L(X)\theta=0\,\},\\
\X_{(\omega,\theta)}(M)&=\{\,X\in\X(M)\mid L_\theta(X)\omega=0\,\},\\
\X_{[\omega]}(M)&=\{\,X\in\X(M)\mid\text{$L(X)\omega=f\omega$,
  $L(X)\theta=-df$ for some $f\in\ca{C}^\infty(M)$}\,\}.
\end{align*}
Clearly, $\X_\omega(M)$ and $\X_{[\omega]}(M)$ are Lie subalgebras,
and it follows from~\eqref{equation;cartan} that
$\X_{(\omega,\theta)}(M)$ is a Lie subalgebra as well.  Plainly we
have the inclusions
\[
\X_\omega(M)\subseteq\X_{[\omega]}(M),\qquad
\X_{(\omega,\theta)}(M)\subseteq\X_{[\omega]}(M).
\]
The condition $L(X)\theta=0$ in the definition of $\X_\omega(M)$ is
usually superfluous.  Namely if $L(X)\omega=0$, then
\[
0=dL(X)\omega=L(X)d\omega=-L(X)\theta\wedge\omega,
\]
so $L(X)\theta=0$ and $X\in\X_\omega(M)$, provided that $\dim(X)\ge4$.
Similarly, the condition $L(X)\theta=-df$ in the definition of
$\X_{[\omega]}(M)$ is superfluous if $\dim(X)\ge4$.

The \emph{(twisted) Hamiltonian vector field} of a function
$f\in\ca{C}^\infty(M)$ is the vector field $X_f$ defined by
$\iota(X_f)\omega=d_\theta f$.  In conformal symplectic geometry
Hamiltonian vector fields do in general not preserve the conformal
symplectic form.  Instead we have $L_\theta(X_f)\omega=d_\theta^2f=0$,
so $X_f\in\X_{(\omega,\theta)}(M)$.  In other words, the map $f\mapsto
X_f$ is a map
\begin{equation}\label{equation;lie-hom}
\ca{C}^{\infty}(M)\longto\X_{(\omega,\theta)}(M).
\end{equation}
An important role is played by the Hamiltonian vector field of the
constant function $1$, which we will call the \emph{(symplectic) Lee
  vector field} and denote by $A=X_1$.  It is easy to show that
\[
\iota(A)\theta=0,\qquad L(A)\theta=0,\qquad L(A)\omega=0.
\]
As in symplectic geometry, the \emph{bracket} of two functions $f$,
$g\in\ca{C}^\infty(M)$ defined by
\[\{f,g\}=\omega(X_f,X_g)\]
makes $\ca{C}^\infty(M)$ a Lie algebra, and the
map~\eqref{equation;lie-hom} is a Lie algebra homomorphism,
i.e.\ $X_{\{f,g\}}=[X_f,X_g]$.  We will repeatedly use the following
elementary result proved by Vaisman.

\begin{proposition}[{\cite[Proposition~2.1]{Va85}}]
\label{proposition;strict}
Let $(M,\omega,\theta)$ be a connected strict conformal symplectic
manifold.  Then $d_\theta\colon\ca{C}^\infty(M)\to\Omega^1(M)$ is
injective.  Hence the homomorphism~\eqref{equation;lie-hom} is
injective and $H_\theta^0(M)=0$.
\end{proposition}

In contrast to the symplectic case the bracket $\{\cdot,\cdot\}$ does
not satisfy the Leibniz identity and so is not a Poisson bracket
(cf.\ \cite[Section 2]{Va85}).  The deficiency of the Leibniz identity
is also manifested in the fact that $d_\theta$ is not a derivation.
In fact, we have $[\Pi,\Pi]=2A\wedge\Pi$, where $\Pi$ is the
$2$-vector field dual to $\omega$, and so the pair $(\Pi,A)$ defines a
Jacobi structure on $M$ in the sense of Lichnerowicz~\cite{Li78}.  It
was proved in~\cite{GL84} (see also~\cite[\S\,3]{DLM91}) that
conformal symplectic structures are precisely the even-dimensional
transitive Jacobi structures, whereas contact forms are the
odd-dimensional transitive Jacobi structures.

%%%%%%%%%%%%%%%%%%%%%%%%%%%%%%%%%%%%%%%%%%%%%%%%%%%%%%%%%%%%%%%%%%%%%%%%
\section{Hamiltonian actions and the local model}
\label{section;hamilton}
%%%%%%%%%%%%%%%%%%%%%%%%%%%%%%%%%%%%%%%%%%%%%%%%%%%%%%%%%%%%%%%%%%%%%%%%

In this section $(M,\omega,\theta)$ denotes a connected conformal
symplectic manifold on which the compact connected Lie group $K$ acts
smoothly.  We introduce Hamiltonian Lie group actions and actions of
Lee type following~\cite{HR01} and~\cite{BGP18} (with some minor
changes in terminology) and prove a ``local normal form'' theorem for
the moment map, which is one of the ingredients in the proof of the
convexity theorem.

\begin{definition}
The $K$-action on $M$ is \emph{weakly Hamiltonian} if there exists a
map $\Phi\colon M\to\lie{k^*}$, called a \emph{moment map}, such that
$\iota(\xi_M)\omega=d_\theta\Phi^\xi$ for all $\xi\in\lie{k}$.  Here
$\Phi^\xi$ is defined by $\langle\Phi(x),\xi\rangle=\Phi^\xi(x)$.  The
$K$-action is \emph{Hamiltonian}, or $M$ is a \emph{Hamiltonian
  conformal symplectic $K$-manifold}, if there exists a moment map
$\Phi$ which is $K$-equivariant with respect to the given action on
$M$ and the coadjoint action on $\lie{k}^*$, i.e.\ $\Phi(g\cdot
x)=\Ad^*_{g^{-1}}(\Phi(x))$ for all $g\in K$.
\end{definition}

If $\Phi$ is a moment map for the $K$-action on $M$, then
$\xi_M=X_{\Phi^\xi}$, and hence $\xi_M\in\X_{(\omega,\theta)}(M)$, for
all $\xi\in\lie{k}$.  Define
$\Phi^\vee\colon\lie{k}\to\ca{C}^\infty(M)$ by
$\Phi^\vee(\xi)=\Phi^\xi$.  The map $\Phi^\vee$ is a lifting of the
infinitesimal action $\lie{k}\to\X(M)$ to a linear map
$\lie{k}\to\ca{C}^\infty(M)$ along the map~\eqref{equation;lie-hom},
\[
\begin{tikzcd}
&\ca{C}^\infty(M)\ar[d]\\
\lie{k}\ar[ur,pos=0.6,dotted,"\Phi^\vee"]\ar[r]&
\X_{(\omega,\theta)}(M).
\end{tikzcd}
\]
The action is Hamiltonian if and only if such a lifting exists.  It
follows from Proposition~\ref{proposition;strict} that if
$(M,\omega,\theta)$ is strict (which is the case of most interest to
us), then the moment map $\Phi$ is uniquely determined by the
$K$-action.  For Hamiltonian $K$-actions on \emph{symplectic}
manifolds the moment map is unique up to an additive constant, and it
is well-known that the constant can be chosen so that the moment map
is equivariant with respect to the coadjoint action on $\lie{k}^*$.
However, in the conformal symplectic setting the moment map is not
necessarily $K$-equivariant.  The following proposition, which is due
to~\cite{HR01} and~\cite{BGP18}, shows that the failure of the
$K$-equivariance lies in the fact that Hamiltonian vector fields do
not preserve the conformal symplectic form.  We review the argument,
because we will need a generalization of the statement to degenerate
$2$-forms in Appendix~\ref{section;conformal-presymplectic}.

\begin{proposition}%
[{\cite[Proposition~2]{HR01}, \cite[Proposition~4.6]{BGP18}}]
\label{proposition;equivariant-moment}
Suppose that $(M,\omega,\theta)$ is strict conformal symplectic and
that the $K$-action is weakly Hamiltonian with moment map $\Phi$.
Then the following conditions are equivalent:
\begin{enumerate}
\item\label{item;equiv}
The action is Hamiltonian.
\item\label{item;inv}
$\omega$ is $K$-invariant.
\item\label{item;bas}
The Lee form $\theta$ is $K$-basic.
\end{enumerate}
\end{proposition}

\begin{proof}
The vector field $\xi_M$ is Hamiltonian, so $L_\theta(\xi_M)\omega=0$
for all $\xi\in\lie{k}$, i.e.\ $L(\xi_M)\omega=\gamma_\xi\omega$,
where $\gamma_\xi=-\theta(\xi_M)\in\ca{C}^\infty(M)$.  It follows that
$\omega$ is conformally $K$-invariant in the sense that
$g^*\omega=c_g\omega$ for all $g\in K$.  Here
$c_g\in\ca{C}^\infty(M,\R_{>0})$ is defined by
\[
c_g(x)=\exp\int_0^1\gamma_\xi\bigl(\exp_K(t\xi_M)\cdot x\bigr)\,dt
\]
for $x\in M$ and for any $\xi\in\lie{k}$ with $g=\exp_K(\xi)$.  Since
$H_\theta^0(M)=0$ by Proposition~\ref{proposition;strict}, we have
\[H^1(\lie{k},H_\theta^0(M))=H^2(\lie{k},H_\theta^0(M))=0,\]
so it follows from \cite[Proposition~3]{HR01} that the moment map
$\Phi$ is unique and has the conformal equivariance property
$\Phi(g\cdot x)=c_g(x)\Ad^*_{g^{-1}}(\Phi(x))$ for all $g\in K$.  The
equivalence \eqref{item;equiv}$\iff$\eqref{item;inv} is now immediate,
because $\omega$ is $K$-invariant if and only if $c_g$ is constant
equal to $1$ for all $g$.  Also $\omega$ is invariant if and only if
$L(\xi_M)\omega=0$, i.e.\ $\gamma_\xi=0$ for all $\xi$,
so~\eqref{item;inv} is equivalent to $\theta$ being horizontal.  Every
closed horizontal form is invariant, so $\theta$ is horizontal if and
only if it is basic.
\end{proof}

\begin{remark}
For Hamiltonian actions on symplectic manifolds, it is well-known that
each component of the moment map is a Morse-Bott function.  However,
this fails in general for Hamiltonian actions on conformal symplectic
manifolds.
\end{remark}

Let $m\in M$.  Recall that the orbit $K\cdot m$ has an
\emph{equivariant tubular neighborhood}, i.e.\ an open neighborhood of
the form $V=K\cdot S$, where $S$ is a slice at $m$ for the $K$-action.
A useful consequence of
Proposition~\ref{proposition;equivariant-moment} is that such a
neighborhood $V$ is conformally equivalent to a \emph{symplectic}
Hamiltonian $K$-manifold.

\begin{proposition}\label{proposition;local-model}
Suppose that $(M,\omega,\theta)$ is strict conformal symplectic and
that the $K$-action is Hamiltonian with moment map $\Phi$.  Let $m\in
M$ and let $V$ be an equivariant tubular neighborhood of $K\cdot m$.
\begin{enumerate}
\item\label{item;function}
There exists a $K$-invariant smooth function $f$ on $V$ such that
$\theta|_V=df$.
\item\label{item;symplectic}
The form $\Omega=e^f\cdot\omega|_V$ is symplectic and the $K$-action
on $(V,\Omega)$ is Hamiltonian with equivariant moment map
$\Psi=e^f\cdot\Phi|_V$.
\end{enumerate}
\end{proposition}

\begin{proof}
\eqref{item;function}~By
Proposition~\ref{proposition;equivariant-moment} the form
$\theta|_V\in\Omega^1(V)$ is basic.  A theorem of Koszul
(see~\cite[p.~141]{K53} or~\cite[Appendix~B3]{GGK02}) says that the
basic de Rham complex $\Omega^*_{K,\bas}(V)$ is acyclic.  Hence
$\theta|_V=df$ for some $K$-invariant $f\in\ca{C}^\infty(V)$.

\eqref{item;symplectic}~The $2$-form $\omega|_V$ is global conformal
symplectic with Lee form $df$.  It follows that
$\Omega=e^f\cdot\omega|_V$ is a symplectic form.  Since $f$ is
$K$-invariant, the form $\Omega$ is $K$-invariant and the map $\Psi$
is $K$-equivariant.  For all $\xi\in\lie{k}$ we have
\[
d\Psi^\xi=d\bigl(e^f\Phi^\xi|_V\bigr)=
e^f\bigl(d\Phi^\xi+\theta\cdot\Phi^\xi\bigr)\big|_V=
e^fd_\theta\Phi^\xi|_V=e^f\iota(\xi_V)\omega|_V=\iota(\xi_V)\Omega,
\]
which shows that $\Psi$ is a moment map for the $K$-action on
$(V,\Omega)$.
\end{proof}

\begin{definition}
The $K$-action on $M$ is of \emph{Lee type} if the Lee vector field
$A$ is generated by the infinitesimal action of a Lie algebra element,
i.e.\ $A=\zeta_M$ for some $\zeta\in\lie{k}$.  Such an element $\zeta$
is a \emph{Lee element} for the $K$-action.
\end{definition}

The Lee type condition is equivalent to the moment map image being
contained in an affine hyperplane.

\begin{lemma}\label{lemma;affine-plane}
Let $(M,\omega,\theta)$ be strict conformal symplectic.  The
$K$-action on $M$ is of Lee type, with Lee element $\zeta\in\lie{k}$,
if and only if the image of $\Phi$ lies in the affine hyperplane
\[\H_\zeta=\{\,v\in\lie{k}^*\mid\inner{v,\zeta}=1\,\}.\]
\end{lemma}

\begin{proof}
By definition $d_\theta\Phi^\zeta=\iota(\zeta_M)\omega$.  If
$\Phi(M)\subseteq\H_\zeta$, then $\Phi^\zeta=1$, so
$\theta=d_\theta1=d_\theta\Phi^\zeta=\iota(\zeta_M)\omega$.  This
shows that $A=\zeta_M$, i.e.\ the action is of Lee type with Lee
element $\zeta$.  Conversely, if $\zeta$ is a Lee element, then
$d_\theta\Phi^\zeta=\theta=d_\theta(1)$.  By
Proposition~\ref{proposition;strict}, since $M$ is connected and
$(\omega,\theta)$ is strict, the map
$d_\theta\colon\ca{C}^\infty(M)\to\Omega^1(M)$ is injective and
therefore we get $\Phi^\zeta=1$, i.e.\ $\Phi(M)\subset\H_\zeta$.
\end{proof}

The next result is a version of~\cite[Corollary~4.14]{BGP18}.

\begin{lemma}\label{lemma;lee}
Let $(M,\omega,\theta)$ be strict conformal symplectic.  Suppose that
the $K$-action is weakly Hamiltonian and of Lee type.  Then the
$K$-action is Hamiltonian if and only if there exists a central Lee
element for the $K$-action.  In particular, the action is Hamiltonian
if $K$ is a torus, and cannot be Hamiltonian if $K$ is semisimple.
\end{lemma}  

\begin{proof}
Suppose the action is Hamiltonian.  Then $\omega$ and $\theta$ are
$K$-invariant by Proposition~\ref{proposition;equivariant-moment}, so
the Lee vector field $A$, which is characterized by
$\iota(A)\omega=\theta$, is $K$-invariant.  Choose any $\xi\in\lie{k}$
such that $A=\xi_M$.  Invariance of $A$ yields
$(\Ad_g\xi)_M=g_*(\xi_M)=g_*A=A$ for all $g\in K$.  Let
$\zeta=\int_K\Ad_g\xi\,dg$, where $dg$ denotes normalized Haar measure
on $K$.  Then $\zeta$ is central and
\[\zeta_M=\int_K(\Ad_g\xi)_M\,dg=\int_KA\,dg=A,\]
so $\zeta$ is a Lee element.  Conversely, suppose there exists a
central Lee element $\zeta$.  Then for all $\xi\in\lie{g}$ we have
\begin{align*}
d_\theta\iota(\xi_M)\theta&=L_\theta(\xi_M)\theta\\
&=L_\theta(\xi_M)\iota(\zeta_M)\omega\\
&=[L_\theta(\xi_M),\iota(\zeta_M)]\omega-
\iota(\zeta_M)L_\theta(\xi_M)\omega\\
&=\iota([\xi,\zeta]_M])\omega\\
&=0.
\end{align*}
Here we used the commutation relations~\eqref{equation;cartan}, the
fact that $L_\theta(\xi_M)\omega=0$, which follows from the assumption
that the vector field $\xi_M$ is Hamiltonian, and the fact that
$\zeta$ is central.  By Proposition~\ref{proposition;strict} we
conclude that $\iota(\xi_M)\theta=0$, i.e.\ $\theta$ is basic.  Hence
the action is Hamiltonian by
Proposition~\ref{proposition;equivariant-moment}.  The last assertion
now follows from the observation that a Lee element for the action on
$M$ must be nonzero, because $M$ is strict conformal symplectic.
\end{proof}

%%%%%%%%%%%%%%%%%%%%%%%%%%%%%%%%%%%%%%%%%%%%%%%%%%%%%%%%%%%%%%%%%%%%%%%%
\section{Convexity properties of the moment map}\label{section;convex}
%%%%%%%%%%%%%%%%%%%%%%%%%%%%%%%%%%%%%%%%%%%%%%%%%%%%%%%%%%%%%%%%%%%%%%%%

The purpose of this section is to prove the convexity theorem,
Theorem~\ref{theorem;main}.  The basic idea is to use
Proposition~\ref{proposition;local-model} to obtain local convexity,
and then to apply a local-global principle.

%%%%%%%%%%%%%%%%%%%%%%%%%%%%%%%%%%%%%%%%%%%%%%%%%%%%%%%%%%%%%%%%%%%%%%%%
\subsection*{Notation and conventions}
%%%%%%%%%%%%%%%%%%%%%%%%%%%%%%%%%%%%%%%%%%%%%%%%%%%%%%%%%%%%%%%%%%%%%%%%

Except for Theorem~\ref{theorem;global}, in this section
$(M,\omega,\theta)$ denotes a connected strict conformal symplectic
manifold equipped with a Hamiltonian action of the compact connected
Lie group $K$.  The (unique) equivariant moment map for the action is
denoted by $\Phi\colon M\to\lie{k}^*$.  We assume the action is of Lee
type and choose a central Lee element $\zeta$, i.e.\ an element
$\zeta\in\lie{z}(\lie{k})$ such that $\zeta_M$ is equal to the Lee
vector field $A$.  The existence of such an element is guaranteed by
Lemma~\ref{lemma;lee}.  We let $\H_\zeta$ be the affine hyperplane
\[\H_\zeta=\{\,v\in\lie{k}^*\mid\inner{v,\zeta}=1\,\}.\]
We fix a maximal torus $T$ of $K$ with Lie algebra $\lie{t}$ and a
closed chamber $C$ in the dual space $\lie{t}^*$.  Since the Lee
element $\zeta$ is central, it is contained in the Cartan subalgebra
$\lie{t}$.  Moreover, $\lie{t}$ has a natural complement in $\lie{k}$
(the sum of the root spaces), and therefore $\lie{t}^*$ can be
naturally regarded as a subspace of $\lie{k}^*$.  The inclusion
$C\hookrightarrow\lie{k}^*$ induces a homeomorphism
$C\stackrel{\simeq}\longto\lie{k}^*/\mathrm{Ad}^*(K)$, which we use to
identify $\lie{k}^*/\mathrm{Ad}^*(K)$ with $C$.  We let
$q\colon\lie{k}^*\to C$ be the quotient map and define
$\phi=q\circ\Phi\colon M\to C$ to be the composition of $\Phi$ with
$q$.  The \emph{moment body} $\Delta(M)$ of $M$ is defined by
\[\Delta(M)=\Phi(M)\cap C=\phi(M).\]
%

%%%%%%%%%%%%%%%%%%%%%%%%%%%%%%%%%%%%%%%%%%%%%%%%%%%%%%%%%%%%%%%%%%%%%%%%
\subsection*{Local convexity}
%%%%%%%%%%%%%%%%%%%%%%%%%%%%%%%%%%%%%%%%%%%%%%%%%%%%%%%%%%%%%%%%%%%%%%%%

Let $\zeta^\circ\subset\lie{k}^*$ be the annihilator of the line
$\R\cdot\zeta$ in $\lie{k}$ spanned by the Lee element.  Then
$\zeta^\circ$ is a hyperplane in $\lie{k}^*$, and for any point $m\in
M$ we have $\H_\zeta=\Phi(m)+\zeta^\circ$.  Let $\H_\zeta^+$ be the
open halfspace
\[
\H^+_\zeta= \{\,v\in\lie{k}^*\,|\,\langle v,\zeta\rangle>0\,\}
\]
and define the \emph{rescaling map}
$\varrho\colon\H^+_\zeta\to\H_\zeta$ by
\begin{equation}\label{equation;rescale}
\varrho(v)=\frac{v}{\inner{v,\zeta}};
\end{equation}
in other words $\varrho(v)$ is the unique intersection point of the
line through $v\in\H_\zeta^+$ with the affine hyperplane $\H_\zeta$.
The rescaling map extends to a projective linear map
\[
\hat{\varrho}\colon\bb{P}(\lie{k}^*\oplus\R)\longto
\bb{P}(\lie{k}^*\oplus\R)
\]
given by $\hat{\varrho}([v,t])=[v,\inner{v,\zeta}]$, where $[v,t]$
denotes the line through a nonzero point $(v,t)\in\lie{k}^*\oplus\R$.
The map $\hat\varrho$ is a projection onto the hyperplane
$\{\,[v,t]\mid\inner{v,\zeta}=t\,\}$, and it is $K$-equivariant, where
we let $K$ act by the coadjoint action on $\lie{k}^*$ and trivially on
$\R$.  These facts imply that $\varrho$ has the following properties.

\begin{lemma}\phantomsection\label{lemma;rescale}
\begin{enumerate}
\item\label{item;open-connected}
The map $\varrho$ is open and its fibres are connected.
\item\label{item;convex-ray}
The map $\varrho$ maps convex sets to convex sets and convex
polyhedral sets to convex polyhedral sets.  Let $\ell=\{\,tv\mid
t>0\,\}$ be the ray (open halfline) through any $v\in\H_\zeta^+$.
Then $\varrho(\ell\cap\H_\zeta^+)$ is a ray in $\H_\zeta$.
\item\label{item;equivariant-chamber}
The affine hyperplane $\H_\zeta$ and the halfspace $\H_\zeta^+$ are
$\Ad^*$-invariant.  The map $\varrho$ is $\Ad^*$-equivariant and maps
$\H_\zeta^+\cap C$ onto $\H_\zeta\cap C$.  The diagram
\[
\begin{tikzcd}[row sep=large]
\H_\zeta^+\ar[r,"\varrho"]\ar[d,"q"']&\H_\zeta\ar[d,"q"]
\\
\H_\zeta^+\cap C\ar[r,"\varrho"]&\H_\zeta\cap C
\end{tikzcd}
\]
commutes.
\end{enumerate}
\end{lemma}

We deduce from this the following local convexity theorem.

\begin{theorem}\label{theorem;local-convexity}
Assume that $(M,\omega,\theta)$ is strict conformal symplectic and
that the $K$-action on $M$ is Hamiltonian and of Lee type.  Then for
each $m\in M$ there exist a closed convex cone $C_m$ in
$\lie{t}^*\cap\H_\zeta$ with apex $\phi(m)$, and a basis of
$K$-invariant open neighborhoods $U$ of $m$ such that
\begin{enumerate}
\item\label{item;fibre-connected}
the fibres of the map $\phi|_U$ are connected;
\item\label{item;open}
$\phi\colon U\to C_m$ is an open map.
\end{enumerate}
\end{theorem}

\begin{proof}
Let $m\in M$ and let $V$ be an equivariant tubular neighborhood of the
orbit $K\cdot m$.  Choose a function $f\in\ca{C}^\infty(V)$ as in
Proposition~\ref{proposition;local-model}, and let
$\Omega=e^f\cdot\omega|_V$ and $\Psi=e^f\cdot\Phi|_V$.  Then
$(V,\Omega)$ is a symplectic Hamiltonian $K$-manifold with moment map
$\Psi$.  Let $x\in V$.  Since $\Phi(x)\in\H_\zeta$ (by
Lemma~\ref{lemma;affine-plane}) and $\Psi(x)$ is a positive scalar
multiple of $\Phi(x)$, we have $\Psi(x)\in\H_\zeta^+$ and therefore
$\Phi(x)=\varrho(\Psi(x))$.  This shows that
$\Phi|_V=\varrho\circ\Psi$.  Putting $\psi=q\circ\Psi$ we obtain from
Lemma~\ref{lemma;rescale}\eqref{item;equivariant-chamber} that
\begin{equation}\label{equation;varrho-phi}
\phi|_V=q\circ\Phi|_V=q\circ\varrho\circ\Psi=\varrho\circ
q\circ\Psi=\varrho\circ\psi,
\end{equation}
and in particular $\phi(V)=\varrho\circ\psi(V)$ (see
Figure~\ref{figure;rescale}).
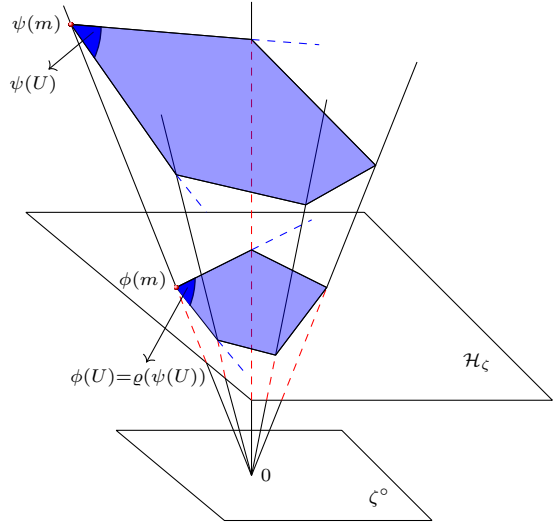
\begin{figure}[h]
\centering
\begin{tikzpicture}[scale=1]
\draw[thick,gray]
%origin
(0,0) coordinate (o) (-1.75,0.6) coordinate (a) (-1.8,0.6) coordinate
(a) (1.2,0.6) coordinate (b) (2.4,-0.6) coordinate (c)
%vertexes of first plane
(-0.36,-0.6) coordinate (d) (-3,3.5) coordinate (e) (1.5,3.5)
coordinate (f) (4,1) coordinate (g)
%vertexes of second plane
(0,1) coordinate (h) (-0.6,1.5) coordinate (l11) (-1,2.5) coordinate
(l12) (-2.4,6) coordinate (l13)
%1st line
(-2.5,6.25) coordinate (l14) (0,1) coordinate (l21) (0,3) coordinate
(l22) (0,5.8) coordinate (l23)
%2nd line
(0,6.3) coordinate (l24) (0.4,1) coordinate (l31) (1,2.5) coordinate
(l32) (1.65,4.125) coordinate (l33)
%3th line
(2.2,5.5) coordinate (l34) (0.2,1) coordinate (l41) (0.32,1.6)
coordinate (l42) (0.72,3.6) coordinate (l43)
%4th line
(1,5) coordinate (l44) (-0.35,1.3) coordinate (l51) (-0.45,1.8)
coordinate (l52) (-1,4) coordinate (l53)
%5th line
(-1.2,4.8) coordinate (l54) (0.8,3.4) coordinate (l221) (-0.1,1.38)
coordinate (l521) (0.9,5.73) coordinate (l231) (-0.6,3.5) coordinate
(l531) (2,0) coordinate (n1) (2.5,2) coordinate (n2) (-2.5,2)
coordinate (n3) (-1.5,6) coordinate (n4) ;
%draw two planes
\draw (o) node[right,black] {$\scriptstyle0$};
\draw (a) -- (b) (b) -- (c)(c) -- (d) (d) -- (a) (e)--(f) (f)--(g)
(g)--(h) (h)--(e);
%draw 1st lines
\draw (o) -- (l11) (l12)--(l13) (l13) --(l14);
\draw[red,dashed](l11) -- (l12);
%draw 2st lines
\draw (o) -- (l21) (l23) --(l24);
\draw[red,dashed](l21) -- (l22)(l22)--(l23);
% draw 3th lines
\draw (o) -- (l31) (l32)--(l33) (l33) --(l34);
\draw[red,dashed](l31) -- (l32);
%draw 4th lines
\draw (o) -- (l41) (l42)--(l43) (l43) --(l44);
\draw[red,dashed](l41) -- (l42);
%draw 5th lines
\draw (o) -- (l51) (l52)--(l53) (l53) --(l54);
\draw[red,dashed](l51) -- (l52);
%draw pentagons
\draw (l12)--(l22) (l22)--(l32) (l32)--(l42) (l42)--(l52)(l52)--(l12)
(l13)--(l23) (l23)--(l33) (l33)--(l43) (l43)--(l53)(l53)--(l13);
\draw[blue,dashed](l23)--(l231)(l53)--(l531)(l22)--(l221)(l52)--(l521);
\filldraw[blue,line width=0.5,draw=black](l12)--(-0.75,2.625) arc
(0:-21.5:1);
\filldraw[blue,line width=0.5,draw=black](l13)--(-2,5.968) arc
(0:-24.6:1);
\node (n1) at (1.7,-0.3) {$\scriptstyle\zeta^{\circ}$};
\node (n2) at (3,1.5) {$\scriptstyle\H_\zeta$};
\node (n3) at (-1.5,1.3) {$\scriptstyle\phi(U)=\varrho(\psi(U))$};
\node (n4) at (-2.9,5.2) {$\scriptstyle\psi(U)$};
\draw (l13) node[left] {$\scriptstyle\psi(m)$} (-1,2.6) node[left]
      {$\scriptstyle\phi(m)$};
\draw[->] (-0.85,2.5)--(-1.4,1.5); \draw[->] (-2.1,5.9)--(-2.7,5.4);
\filldraw[fill=blue,even odd rule,fill
  opacity=0.4](l12)--(l22)--(l32)--(l42)--(l52);
\filldraw[fill=blue,even odd rule,fill
  opacity=0.4](l13)--(l23)--(l33)--(l43)--(l53)--(l13);
\shade[ball color=red] (l12) circle (0.03);
\shade[ball color=red] (l13) circle (0.03);
\end{tikzpicture}
\caption{Projecting the local moment cone}\label{figure;rescale}
\end{figure}
The symplectic version of the local convexity
theorem~\cite[Theorem~6.5]{Sj98} asserts that there exist a basis of
$K$-invariant open neighborhoods $U\subset V$ of $m$ and a rational
convex polyhedral cone $\Delta_m\subset\lie{t}^*\subset\lie{k}^*$ with
apex $\psi(m)$ such that $\psi|_U\colon U\to\Delta_m$ is an open map
with connected fibres.  In particular, the set $\psi(U)$ is a
neighborhood of $\psi(m)$ in $\Delta_m$ and the cone with apex
$\psi(m)$ spanned by $\psi(U)$ is equal to $\Delta_m$.  Define
\[C_m=\varrho(\Delta_m\cap\ca{H}^+_\zeta).\]
Lemma~\ref{lemma;rescale}\eqref{item;convex-ray} shows that $C_m$ is a
closed convex polyhedral cone, namely the cone with apex
$\varrho(\psi(m))=\phi(m)$ spanned by $\varrho(\psi(U))$.
Lemma~\ref{lemma;rescale}\eqref{item;open-connected} shows that
$\varrho\colon\Delta_m\cap\H^+_\zeta\to C_m$ is an open map (for the
subspace topologies inherited from $\H^+_\zeta$ and $\H_\zeta$,
respectively).  It now follows from~\eqref{equation;varrho-phi} that
$\phi\colon U\to C_m$ is an open map, which proves~\eqref{item;open}.

Let $x\in U$ and put $a=\phi(x)$.  Put $A=\varrho^{-1}(a)$; then $A$
is connected by Lemma~\ref{lemma;rescale}\eqref{item;open-connected}
and $(\phi|_U)^{-1}(a)=(\psi|_U)^{-1}(A)$
by~\eqref{equation;varrho-phi}.  As the map $\psi\colon
U\to\Delta_m\cap\H^+_\zeta$ is open, so is the surjective map
\[\psi\colon\psi^{-1}(A)\to A.\]
Now suppose that $(\phi|_U)^{-1}(a)=(\psi|_U)^{-1}(A)$ was not
connected.  Then it is the union of two disjoint nonempty open subsets
$W_1$ and $W_2$.  Therefore the sets $\psi(W_1)$ and $\psi(W_2)$ form
an open covering of $A$.  Thanks to the connectedness of $A$, the
intersection $B=\psi(W_1)\cap\psi(W_2)$ is nonempty and therefore
$(\psi|_U)^{-1}(b)$ is not connected for any $b\in B$.  This
contradicts the fact that the fibres of the map $\psi\colon
U\to\Delta_m$ are connected, which completes the proof
of~\eqref{item;fibre-connected}.
\end{proof}

%%%%%%%%%%%%%%%%%%%%%%%%%%%%%%%%%%%%%%%%%%%%%%%%%%%%%%%%%%%%%%%%%%%%%%%%
\subsection*{From local to global}
%%%%%%%%%%%%%%%%%%%%%%%%%%%%%%%%%%%%%%%%%%%%%%%%%%%%%%%%%%%%%%%%%%%%%%%%

We will now derive the global convexity theorem by means of a
local-global principle.  We use the version of the local-global
principle due to \cite{HNP94}.  Let $X$ be a connected Hausdorff
topological space and $E$ a finite dimensional vector space.  A
continuous map $\Psi\colon X\to E$ is said to be \emph{locally fiber
  connected}, if every point $x$ in $X$ admits arbitrarily small
neighborhoods $U$ such that $\Psi^{-1}(\Psi(u))\cap U$ is connected
for all $u\in U$.  We say that a map $x\mapsto C_x$ assigning to each
point $x\in X$ a closed convex cone $C_x\subset E$ with apex $\Psi(x)$
is a system of \emph{local convexity data} if the following conditions
hold:
\begin{enumerate}
\item
for each $x$ there exists an arbitrarily small open neighborhood $U_x$
such that the map $\Psi\colon U_x\to C_x$ is open;
\item
$\Psi^{-1}(\Psi(u))\cap U_x$ is connected for all $u\in U_x$.
\end{enumerate}

\begin{theorem}[{\cite[Theorem 3.10]{HNP94}}]\label{l-g-p}
Suppose that $\Psi\colon X\to E$ is a proper, locally fiber connected
map with the local convexity data $(C_x)_{x\in X}$.  Then the fibres
of $\Psi$ are connected and $\Psi\colon X\to\Psi(X)$ is an open map;
moreover, $\Psi(X)$ is a closed convex polyhedral subset of $E$.
\end{theorem}

\begin{proof}[Proof of Theorem~\ref{theorem;main}]
Theorem~\ref{theorem;local-convexity} shows that the family of cones
$(C_m)_{m\in M}$ is a system of local convexity data for the map
$\phi\colon M\to C$.  On account of the properness of the moment map
$\Phi$ and the quotient map $q$, the map $\phi=q\circ\Phi\colon M\to
C$ is proper.  Theorem~\ref{theorem;main} now follows immediately from
Theorem~\ref{l-g-p}.
\end{proof}

\begin{remarks}
\begin{numerate}
\item\label{item;cones}
Recall that a closed convex subset of a finite dimensional vector
space can be expressed as the intersection of all its supporting
halfplanes.  As shown in~\cite[\S\,1]{HNP94}), we have
\[C_m=\phi(m)+L_{\phi(m)}(\phi(M)),\]
where $L_{\phi(m)}(\phi(M))$ is the intersection of all supporting
halfplanes of $\phi(M)$ at the point $\phi(m)$.  It follows that
\[\Delta(M)=\bigcap_{m\in M}C_m,\]
which is an intersection of a locally finite family of convex cones.
\item\label{item;conformal}
Let $a\in\ca{C}^\infty(M,\R_{>0})$.  A gauge transformation
$a\cdot(\omega,\theta)$ changes the moment map to $a\Phi$, which
usually destroys the convexity of the image.  Thus the moment body
$\Delta(M)$ is not a conformal invariant of $M$.
\item\label{item;rational}
The above proof does not tell us whether the moment body $\Delta(M)$
is semirational or not.  The reason lies in the fact the the rescaling
map does not preserve the rationality of convex cones in general.
However, for strict conformal symplectic manifolds of the first kind
we will state a necessary and sufficient condition for the moment body
to be semirational; see Theorem~\ref{theorem;conformal-first}.
\end{numerate}
\end{remarks}

We finish this section with a counterpart of
Theorem~\ref{theorem;main} for global conformal symplectic manifolds.

\begin{theorem}\label{theorem;global}
Let $(M,\omega,\theta=df)$ be a connected global conformal symplectic
manifold equipped with a $K$-action which leaves the forms $\omega$
and $\theta$ invariant.  Suppose that the action is Hamiltonian and
that the moment map $\Phi$ for the action is proper.  Suppose also
that the image of $\Phi$ is contained in the hyperplane $\H_\zeta$ for
some $\zeta\in\lie{k}$.  Then the conclusions of the convexity
theorem, Theorem~\ref{theorem;main}, hold for $M$.
\end{theorem}

\begin{proof}
As in the proof of Lemma~\ref{lemma;affine-plane} one sees that the
action is of Lee type.  As in the proof of Lemma~\ref{lemma;lee} one
shows that the Lee element $\zeta$ can be assumed to be central.  The
remainder of the proof is identical to that of
Theorem~\ref{theorem;main}.
\end{proof}

The assumptions of this theorem apply for instance if $M$ is a closed
$K$-invariant global conformal symplectic submanifold of a strict
conformal symplectic manifold furnished with a Hamiltonian $K$-action.

%%%%%%%%%%%%%%%%%%%%%%%%%%%%%%%%%%%%%%%%%%%%%%%%%%%%%%%%%%%%%%%%%%%%%%%%
\section{Symplectic, conformal  symplectic, and presymplectic convexity}
\label{section;conformal-pre}
%%%%%%%%%%%%%%%%%%%%%%%%%%%%%%%%%%%%%%%%%%%%%%%%%%%%%%%%%%%%%%%%%%%%%%%%

In this section we explain how the conformal symplectic convexity
theorem of the previous section relates to Kirwan's symplectic
convexity theorem~\cite{Ki84} as well as to the presymplectic
convexity theorem of~\cite{LS17}.

The relationship to symplectic convexity arises from the fact that a
conformal symplectic manifold $M$ has symplectic covering manifolds
$\tilde{M}$, and that a Hamiltonian $K$-action on $M$ lifts to a
Hamiltonian action of a covering group $\tilde{K}$ on $\tilde{M}$.
Theorems~\ref{theorem;rank-1} and~\ref{theorem;rank-2} below assert
that in good cases the moment body $\Delta(\tilde{M})$ is a cone over
the moment body $\Delta(M)$, and that $\Delta(M)$ is the intersection
of $\Delta(\tilde{M})$ with an affine hyperplane.  These facts
generalize results for torus actions obtained in~\cite{BGP18}.

The relationship to presymplectic convexity comes from the fact that
the Lee form of a conformal symplectic manifold $M$ defines a singular
foliation, which we will call the \emph{Lee foliation}, whose regular
leaves are of codimension $1$ and carry natural presymplectic
structures.  The $K$-action preserves each leaf $L$, provided that the
action is of Lee type.  We will see that the Lee type condition is
equivalent to a leafwise transitivity condition for the action on $L$.
Theorem~\ref{theorem;lee-leaf} below asserts a form of Reeb stability
for the Lee foliation and also that the moment body $\Delta(L)$ is
equal to $\Delta(M)$ for all leaves $L$.  In particular all leaves
have the same moment body.  This fact was observed for Vaisman
manifolds in~\cite{BGP18}.

%%%%%%%%%%%%%%%%%%%%%%%%%%%%%%%%%%%%%%%%%%%%%%%%%%%%%%%%%%%%%%%%%%%%%%%%
\subsection*{Notation and conventions}
%%%%%%%%%%%%%%%%%%%%%%%%%%%%%%%%%%%%%%%%%%%%%%%%%%%%%%%%%%%%%%%%%%%%%%%%

In this section $(M,\omega,\theta)$ denotes a connected strict
conformal symplectic manifold equipped with a Hamiltonian action of
the compact connected Lie group $K$, which has equivariant moment map
$\Phi\colon M\to\lie{k}^*$.  The moment body of $M$ is denoted by
$\Delta(M)=\Phi(M)\cap C$, where $C$ is a fixed closed Weyl chamber in
$\lie{t}^*$, the dual of the Lie algebra $\lie{t}$ of a maximal torus
$T$ of $K$.

%%%%%%%%%%%%%%%%%%%%%%%%%%%%%%%%%%%%%%%%%%%%%%%%%%%%%%%%%%%%%%%%%%%%%%%%
\subsection*{Definitions}
%%%%%%%%%%%%%%%%%%%%%%%%%%%%%%%%%%%%%%%%%%%%%%%%%%%%%%%%%%%%%%%%%%%%%%%%

A \emph{presymplectic form} on a manifold $P$ is a closed $2$-form
$\sigma\in\Omega^2(P)$ of constant rank.  The kernel of a
presymplectic form $\sigma$ on $P$ is an involutive subbundle of $TP$,
which generates a (regular) foliation $\F_\sigma$ called the
\emph{null foliation} of $\sigma$.

Let $\F$ be a foliation of a manifold $P$.  Let $\X(\F)$ be the Lie
algebra of vector fields tangent to $\F$.  A vector field $X\in\X(M)$
is \emph{foliate} if $[X,Y]\in\X(\F)$ for all $Y\in\X(\F)$.  A smooth
map $\phi\colon X\to X$ is \emph{foliate} if $\phi(L)$ is contained in
a leaf of $\F$ for every leaf $L$ of $\F$.  A foliate vector field
generates a flow consisting of (local) foliate diffeomorphisms.
Suppose that a Lie group $G$ acts smoothly on $P$ by foliate
transformations.  Let $\lie{n}_\F$ be the set of all
$\xi\in\lie{g}=\Lie(G)$ such that the induced vector field $\xi_P$ is
tangent to $\F$.  Then $\lie{n}_\F$ is an ideal of $\lie{g}$ called
the \emph{null ideal} of the action.  (See~\cite[\S\,2.5]{LS17}.)  The
\emph{null subgroup} is the immersed (not necessarily closed) normal
subgroup $N_\F$ of $G$ generated by $\exp(\lie{n}_\F)$.  The
$G$-action on $(P,\F)$ is \emph{clean} if
\[T_p(G\cdot p)\cap T_p\F=T_p(N_\F\cdot p)\]
for every $p\in P$.  The $G$-action is \emph{leafwise transitive} if
$T_p\F=T_p(N_\F\cdot p)$ for every $p\in P$.  Clearly, if the action
is leafwise transitive, it is clean.

%%%%%%%%%%%%%%%%%%%%%%%%%%%%%%%%%%%%%%%%%%%%%%%%%%%%%%%%%%%%%%%%%%%%%%%%
\subsection*{Presentations}
%%%%%%%%%%%%%%%%%%%%%%%%%%%%%%%%%%%%%%%%%%%%%%%%%%%%%%%%%%%%%%%%%%%%%%%%

See~\cite[\S\,2.1]{BGP18} for the following definitions and facts.  A
\emph{presentation} of the strict conformal symplectic manifold
$(M,\omega,\theta)$ is a connected symplectic manifold
$(\tilde{M},\tilde{\omega})$ equipped with a Galois covering map
\[p\colon\tilde{M}\longto M\]
with covering group $\Gamma$, a group homomorphism
$\chi\colon\Gamma\to(\R,+)$, and a function
$\tilde{f}\in\ca{C}^\infty(\tilde{M})$ called the \emph{potential} of
the presentation, which is required to have the following properties:
$0$ is a regular value of $\tilde{f}$;
\[
\tilde\omega=e^{\tilde{f}}p^*\omega;\qquad d\tilde{f}=p^*\theta;\qquad
\text{and}\quad\gamma^*\tilde{f}=\tilde{f}+\chi(\gamma)\quad\text{for
  all $\gamma\in\Gamma$}.
\]
Given a presentation $\tilde{M}$, the potential $\tilde{f}$ is unique
up to an additive constant $c$ and the symplectic form
$\tilde{\omega}$ is unique up to a multiplicative constant $e^c$.  The
condition that $0$ should be a regular value of $\tilde{f}$ is not
imposed in~\cite{BGP18}, but is easy to fulfil: pick a regular value
$-c$ of $\tilde{f}$ and replace $\tilde{f}$ with $\tilde{f}+c$ and
$\tilde\omega$ with $e^c\tilde\omega$ to make $0$ a regular value.
Presentations of $M$ exist, e.g.\ the universal cover of $M$.  A
presentation $\hM$ of $M$ is \emph{minimal} if for every other
presentation $\tilde{M}$ there is a morphism of covering spaces
$\tilde{M}\to\hM$.  Minimal presentations exist and are unique up to
covering space isomorphisms.  A presentation $\hM$ is minimal if and
only if the corresponding homomorphism $\chi\colon\Gamma\to\R$ is
injective.

Conversely, suppose we are given a connected symplectic manifold
$(\tilde{M},\tilde\omega)$, a discrete group $\Gamma$ acting properly
and freely on $\tilde{M}$, a homomorphism $\chi\colon\Gamma\to\R$, and
a function $\tilde{f}$ such that
$\gamma^*\tilde\omega=e^{\chi(\gamma)}\tilde\omega$ and
$\gamma^*\tilde{f}=\tilde{f}+\chi(\gamma)$ for all $\gamma\in\Gamma$.
Then the forms $e^{-\tilde{f}}\tilde\omega$ and $d\tilde{f}$ are
$\Gamma$-invariant, and therefore descend to forms $\omega$ and
$\theta$ on $M=\tilde{M}/\Gamma$.  The pair $(\omega,\theta)$ is a
conformal symplectic structure on $M$ and $\tilde{M}$ is a
presentation of $M$.

The \emph{period homomorphism} $\per_\theta\colon\pi_1(M)\to\R$ is
defined by $\per_\theta([c])=\int_c\theta$; it depends only on the Lee
class $[\theta]\in H^1(M;\R)$.  The \emph{period group} of $\theta$ is
the subgroup $\Per_\theta=\per_\theta(\pi_1(M))\cong\Z^k$ of $\R$; the
\emph{rank} of $\theta$, or of $M$, is the rank $k$ of the period
group.  We have $\Per_\theta=\chi(\Gamma)$ for any presentation of $M$
and $1\le\rank(M)\le b_1(M)$.

%%%%%%%%%%%%%%%%%%%%%%%%%%%%%%%%%%%%%%%%%%%%%%%%%%%%%%%%%%%%%%%%%%%%%%%%
\subsection*{The moment cone of a presentation}
%%%%%%%%%%%%%%%%%%%%%%%%%%%%%%%%%%%%%%%%%%%%%%%%%%%%%%%%%%%%%%%%%%%%%%%%

Let $(\tilde{M},\tilde\omega,\Gamma,\chi,\tilde{f})$ be a presentation
of $(M,\omega,\theta)$.  The $K$-action on $M$ lifts to an action of
an appropriate covering group $\tilde{K}$ on $\tilde{M}$, which
commutes with the $\Gamma$-action.  By~\cite[Remark~4.4]{BGP18} the
$\tilde{K}$-action on the symplectic manifold
$(\tilde{M},\tilde{\omega})$ is Hamiltonian with equivariant moment
map given by
\[\tilde\Phi=e^{\tilde{f}}p^*\Phi.\]
By~\cite[Lemma~4.2]{BGP18}, if $\tilde{M}$ is minimal we can take
$\tilde{K}=K$.  We denote by
$\Delta(\tilde{M})=\tilde\Phi(\tilde{M})\cap C$ the moment body of
$\tilde{M}$ with respect to the Hamiltonian $\tilde{K}$-action.
Applying a conformal transformation $a\cdot(\omega,\theta)$ to the
conformal symplectic structure, where $a$ is a positive function on
$M$, has the effect of replacing the potential $\tilde{f}$ by
$\tilde{f}-\log p^*a$, but changes neither the symplectic form
$\tilde\omega$ nor the moment map $\tilde\Phi$, and therefore has no
effect on the moment body $\Delta(\tilde{M})$.  On the other hand,
shifting the potential by a constant $c$ has the effect of dilating
the moment body $\Delta(\tilde{M})$ by the positive constant $e^c$.
Up to such dilations, the moment body is independent of the
presentation $\tilde{M}$.  Thus the moment body $\Delta(\tilde{M})$,
unlike the moment body $\Delta(M)$ of $M$ itself, is a conformal
invariant of $M$, up to a positive multiplicative constant $e^c$.  In
fact, the equivariance property
\begin{equation}\label{equation;scale}
\gamma^*\tilde\Phi=e^{\chi(\gamma)}\tilde\Phi
\end{equation}
for all $\gamma\in\Gamma$ shows that $\Delta(\tilde{M})$ is preserved
by all dilations in the subgroup $\exp(\chi(\Gamma))$ of $\R_{>0}$.
Under what conditions is $\Delta(\tilde{M})$ preserved by all
dilations in $\R_{>0}$?  There is a natural dichotomy according to
whether $M$ has rank $1$ (when $\Per_\theta\cong\Z$ is discrete in
$\R$) or rank $\ge2$ (when $\Per_\theta$ is dense in $\R$).  The
following two results extend to the nonabelian case Theorems~5.13
and~5.15 of~\cite{BGP18}.

\begin{theorem}\label{theorem;rank-1}
Let $(M,\omega,\theta)$ be a compact connected strict conformal
symplectic manifold of rank $1$ equipped with a Hamiltonian
$K$-action.  Let $(\tilde{M},\tilde\omega,\Gamma,\chi,\tilde{f})$ be a
presentation of $(M,\omega,\theta)$.
\begin{enumerate}
\item\label{item;1-not}
Suppose that $0\not\in\Delta(M)$, or equivalently
$0\not\in\Delta(\tilde{M})$.  Then $\Delta(\tilde{M})\cup\{0\}$ is a
rational convex polyhedral cone and
$\Delta(\tilde{M})=\R_{>0}\cdot\Delta(M)$.
\item\label{item;1-lee}
Suppose that the $K$-action on $(M,\omega,\theta)$ is of Lee type with
central Lee element $\zeta$.  Then
$\Delta(M)=\Delta(\tilde{M})\cap\H_\zeta$.  Hence the convex polytope
$\Delta(M)$ is semirational if and only if the line $\R\zeta$ in
$\lie{t}^*$ is rational.
\end{enumerate}
\end{theorem}  

\begin{proof}
\eqref{item;1-not}~The moment body $\Delta(\tilde{M})$ being
independent of the presentation up to a dilation, we may, and will,
assume without loss of generality that the presentation $\tilde{M}$ is
minimal.  The assumptions $0\not\in\Phi(M)$, $M$ has rank $1$, and
$\tilde{M}$ is minimal imply that the moment map $\tilde\Phi$ is
proper when viewed as a map $\tilde{M}\to\lie{k}^*\setminus\{0\}$.
(See~\cite[Lemma~5.14]{BGP18}.  This lemma asserts that $\tilde\Phi$
is proper viewed as a map $\tilde{M}\to\lie{k}^*$, but that is
incorrect.)  Therefore, by the version of the Kirwan convexity theorem
due to~\cite{LMTW98}, the moment body $\Delta(\tilde{M})$ is a closed
convex polyhedral subset of the punctured chamber $C\setminus\{0\}$.
It now follows from the scaling property~\eqref{equation;scale} that
$\Delta(\tilde{M})$ is $\R_{>0}$-invariant.  Hence its closure
$\Delta(\tilde{M})\cup\{0\}$ is a closed convex polyhedral cone
contained in $C$.  The facets of $\Delta(\tilde{M})$ have rational
normal vectors, so the cone $\Delta(\tilde{M})\cup\{0\}$ is rational.

\eqref{item;1-lee}~We have $\Phi^\zeta=1$
(Lemma~\ref{lemma;affine-plane}) and
$p^*\Phi=e^{-\tilde{f}}\tilde{\Phi}$, so
\begin{equation}\label{equation;cone-project}
\Delta(M)=\varrho(\Delta(\tilde{M})),
\end{equation}
where $\varrho$ is the projection~\eqref{equation;rescale}.  On the
other hand, $\Phi^\zeta=1$ implies $0\not\in\Phi(M)$, so
$\Delta(\tilde{M})=\R_{>0}\cdot\Delta(M)$ by~\eqref{item;1-not}, and
hence $\Delta(M)=\Delta(\tilde{M})\cap\H_\zeta$.  The last assertion
now follows from~\eqref{item;1-not}.
\end{proof}

If $M$ has rank $\ge2$ the moment map $\tilde\Phi$ is not proper, but
nevertheless we get a slightly weaker conclusion under a stronger
hypothesis.

\begin{theorem}\label{theorem;rank-2}
Let $(M,\omega,\theta)$ be a compact connected strict conformal
symplectic manifold of rank $\ge2$ equipped with a Hamiltonian
$K$-action.  Let $(\tilde{M},\tilde\omega,\Gamma,\chi,\tilde{f})$ be a
presentation of $(M,\omega,\theta)$.
\begin{enumerate}
\item\label{item;2-not}
Suppose that $0\not\in\Delta(M)$.  Suppose also that the Lee form
$\theta$ vanishes nowhere, or equivalently that the potential
$\tilde{f}$ has no singular values.  Then
$\Delta(\tilde{M})=\R_{>0}\cdot\Delta(M)$.
\item\label{item;2-lee}
Suppose that the $K$-action on $(M,\omega,\theta)$ is of Lee type with
central Lee element $\zeta$.  Suppose also that for some
$a\in\ca{C}^\infty(M,\R_{>0})$ the gauge-transformed Lee form
$\theta-d\log a$ vanishes nowhere.  Then $\Delta(\tilde{M})\cup\{0\}$
is a convex polyhedral cone and
$\Delta(M)=\Delta(\tilde{M})\cap\H_\zeta$.
\end{enumerate}
\end{theorem}  

\begin{proof}
\eqref{item;2-not}~The following argument adapts the proof
of~\cite[Theorem~5.15]{BGP18}.  Let $\tilde{x}_0\in\tilde{M}$.
Because the rank of $M$ is $\ge2$, the ray
$\R_{>0}\cdot\tilde{\Phi}(\tilde{x}_0)$ intersects the image
$\tilde{\Phi}(\tilde{M})$ in a dense subset of the ray.  We will argue
that the intersection is an \emph{open} subset of the ray.  Let
$x_0=p(\tilde{x}_0)\in M$.  Let $H$ be the identity component of the
stabilizer $K_{x_0}$ and let
$\lie{h}=\Lie(H)=\Lie(K_{x_0})=\Lie(\tilde{K}_{\tilde{x}_0})$.  Define
\begin{alignat*}{3}
M^{\lie{h}}&=\{\,x\in M\mid\Lie(K_x)\supseteq\lie{h}\,\},&\qquad
M_{\lie{h}}&=\{\,x\in M\mid\Lie(K_x)=\lie{h}\,\},
\\
\tilde{M}^{\lie{h}}&=\{\,\tilde{x}\in\tilde{M}\mid
\Lie(\tilde{K}_{\tilde{x}})\supseteq\lie{h}\,\},&\qquad
\tilde{M}_{\lie{h}}&=\{\,\tilde{x}\in\tilde{M}\mid
\Lie(\tilde{K}_{\tilde{x}})=\lie{h}\,\}.
\end{alignat*}
It is an elementary fact that $M^{\lie{h}}$,
resp.\ $\tilde{M}^{\lie{h}}$, is a closed submanifold of $M$,
resp.\ $\tilde{M}$ (which may have connected components of varying
dimensions), and that $M_{\lie{h}}$, resp.\ $\tilde{M}_{\lie{h}}$, is
an open submanifold of $M^{\lie{h}}$, resp.\ $\tilde{M}^{\lie{h}}$.
Moreover, the submanifold $\tilde{M}^{\lie{h}}$ is symplectic.  It
follows that $M^{\lie{h}}=p(\tilde{M}^{\lie{h}})$ is a closed
conformal symplectic submanifold of $M$ with conformal symplectic
structure $(\omega^{\lie{h}}=\omega|_{M^{\lie{h}}},
\theta^{\lie{h}}=\theta|_{M^{\lie{h}}})$.  If $V$ denotes the tangent
space $T_xM$ at some point $x\in M^{\lie{h}}$, then the tangent space
to $M^{\lie{h}}$ is $T_xM^{\lie{h}}=V^{\lie{h}}$, the subspace of
$\lie{h}$-fixed vectors.  This subspace has a natural complement: we
have $V=V^{\lie{h}}\oplus\lie{h}V$, where $\lie{h}V$ is defined as the
subspace spanned by all vectors $\eta\cdot v=[\eta_M,v]$ with
$\eta\in\lie{h}$ and $v\in V$.  Thus the normal bundle
$N=N_MM^{\lie{h}}$ of $M^{\lie{h}}$ is naturally a subbundle of
$TM|_{M^{\lie{h}}}$.  For $v\in V$ and $\eta\in\lie{h}$ we have
\[
\theta_x(\eta\cdot v)=\omega(A_x,\eta\cdot v)=-\omega(\eta\cdot
A_x,v)=0,
\]
because $\omega$ and the Lee vector field $A$ are $K$-invariant.
Therefore $\theta|_N=0$.  As in~\cite[Corollary~5.10]{BGP18} this
implies that the Lee form $\theta^{\lie{h}}$ is not exact.  For if
$\theta^{\lie{h}}=df$ for some $f\in\ca{C}^\infty(M^{\lie{h}})$, then
$\theta^{\lie{h}}_x=0$ for any $x$ in the compact manifold
$M^{\lie{h}}$ where $f$ attains an extremum.  Hence we would have
$\theta_x=0$, which contradicts the hypothesis on $\theta$.  This
shows that the conformal symplectic manifold $M^{\lie{h}}$ is strict.
Since $\eta_M=0$ on $M^{\lie{h}}$ for all $\eta\in\lie{h}$, it follows
from Proposition~\ref{proposition;strict} that $\Phi^\eta=0$ on
$M^{\lie{h}}$ for all $\eta\in\lie{h}$; in other words
$\Phi(M^{\lie{h}})\subseteq\lie{h}^\circ$.  For $h\in H$ and $x\in
M^{\lie{h}}$ we have $\Phi(x)=\Phi(hx)=\Ad^*_h\Phi(x)$, and therefore
\[
\Phi(M^{\lie{h}})\subseteq\lie{h}^\circ\cap\lie{k}^H.
\]
Let $N_K(H)$ be the normalizer of $H$, $L=N_K(H)/H$, and
$\lie{l}=\Lie(L)$.  The subspace $\lie{h}^\circ\cap\lie{k}^H$ of
$\lie{k}^*$ is naturally isomorphic to $\lie{l}^*$, and $\Phi$
restricts to a moment map for the $L$-action on $M^{\lie{h}}$;
cf.~\cite[\S\,3]{SL91}.  Likewise, $\tilde{\Phi}$ maps
$\tilde{M}^{\lie{h}}$ to $\lie{l}^*$ and the restriction of
$\tilde{\Phi}$ to $\tilde{M}^{\lie{h}}$,
\[
\tilde{\Phi}^{\lie{h}}\colon\tilde{M}^{\lie{h}}\longto\lie{l}^*,
\]
is a moment map for the $L$-action on $\tilde{M}^{\lie{h}}$.  The
$L$-action on the open submanifold $\tilde{M}_{\lie{h}}$ of
$\tilde{M}^{\lie{h}}$ is locally free, so $\tilde{\Phi}^{\lie{h}}$
restricted to $\tilde{M}_{\lie{h}}$ is a submersion.  Since
$\tilde{x}_0\in\tilde{M}_{\lie{h}}$, it follows that the image
$\tilde\Phi(\tilde{M}_{\lie{h}})$ intersects the ray through
$\tilde\Phi(\tilde{x}_0)$ in an open subset.

\eqref{item;2-lee}~Let $\Delta_a(M)$ be the moment body of $M$ with
respect to the moment map $a\Phi$.  It follows from~\eqref{item;2-not}
that
\[
\Delta(\tilde{M})=\R_{>0}\cdot\Delta_a(M)=\R_{>0}\cdot\Delta(\tilde{M}).
\]
Therefore, by~\eqref{equation;cone-project},
$\Delta(M)=\varrho(\Delta(\tilde{M}))=\Delta(\tilde{M})\cap\H_\zeta$.
Hence also $\Delta(\tilde{M})=\R_{>0}\cdot\Delta(M)$.  By
Theorem~\ref{theorem;main}, $\Delta(M)$ is a convex polytope, so
$\Delta(\tilde{M})\cup\{0\}$ is a convex polyhedral cone.
\end{proof}

%%%%%%%%%%%%%%%%%%%%%%%%%%%%%%%%%%%%%%%%%%%%%%%%%%%%%%%%%%%%%%%%%%%%%%%%
\subsection*{The Lee foliation}
%%%%%%%%%%%%%%%%%%%%%%%%%%%%%%%%%%%%%%%%%%%%%%%%%%%%%%%%%%%%%%%%%%%%%%%%

For $x$ and $y\in M$ define $x\sim y$ if there exists a smooth path
$\gamma$ joining $x$ to $y$ and satisfying $\theta(\gamma'(t))=0$ for
all $t$.  The relation $\sim$ is an equivalence relation on $M$, whose
equivalence classes we call the \emph{leaves} of the \emph{Lee
  foliation} $\F_\theta$.  The leaves may be singular at points where
$\theta$ vanishes, but on the open subset
\[U_\theta=\{\,x\in M\mid\theta_x\ne0\,\}\]
the foliation $\F_\theta$ is a regular foliation, whose tangent bundle
is $T\F_\theta=\ker(\theta)$ and whose leaves are immersed
submanifolds of codimension $1$.  For every leaf $L$ the form $\theta$
vanishes on $L\cap U_\theta$, so the $2$-form $\omega_L=\omega|_{L\cap
  U_\theta}$ is closed.  The Lee vector field $A$ satisfies
$\iota(A)\theta=0$, and therefore restricts to a vector field $A_L$ on
$L\cap U_\theta$, which spans the null foliation of $\omega_L$.  Thus
we see that the regular part of each leaf is a presymplectic manifold
of corank $1$.

We assert that a Hamiltonian $K$-action preserves each leaf of the Lee
foliation and that, if the action is of Lee type, the action on each
leaf is leafwise transitive.

\begin{lemma}\label{lemma;lee-orbit}
Let $(M,\omega,\theta)$ be a connected strict conformal symplectic
manifold equipped with a Hamiltonian $K$-action.  Then $K\cdot L=L$
for every leaf $L$ of $\F_\theta$, and the $K$-action on the
presymplectic manifold $L\cap U_\theta$ is Hamiltonian with moment map
$\Phi_L=\Phi|_{L\cap U_\theta}$.  If the action is of Lee type, then
the action on $L\cap U_\theta$ is leafwise transitive relative to the
null foliation of $\omega_L$.
\end{lemma}

\begin{proof}
Let $\xi\in\lie{k}$.  From the fact that $\theta$ is basic
(Proposition~\ref{proposition;equivariant-moment}) it follows that
$\theta(\xi_M)=0$.  Therefore $T_x(K\cdot x)$ is tangent to $L$ for
every $x\in L$, and hence $K\cdot x\subseteq L$.  On $M$ we have
$d_\theta\Phi^\xi=\iota(\xi_M)\omega$, so on $L\cap U_\theta$ we have
$d\Phi_L^\xi=\iota(\xi_L)\omega_L$ for all $\xi\in\lie{k}$.  Thus
$\Phi_L$ is a moment map for the $K$-action on $L\cap U_\theta$.  If
$\zeta$ is a Lee element for the action, the null foliation of
$\omega_L$ is spanned by $A_L=\zeta_L$, so the $K$-action on $L\cap
U_\theta$ is leafwise transitive.
\end{proof}

If $\theta$ vanishes nowhere, all the leaves of $\F_\theta$ are
nonsingular, and we can choose a vector field $B\in\X(M)$ satisfying
$\theta(B)=1$.  Such a vector field is unique up to a vector field
tangent to the Lee foliation and its orbits are transversals to the
Lee foliation.  We claim that $\theta$ is $B$-invariant and that $B$
is foliate.

\begin{lemma}\label{lemma;lee-transverse}
Let $(M,\omega,\theta)$ be a connected strict conformal symplectic
manifold.  Suppose the Lee form $\theta$ vanishes nowhere and choose a
vector field $B$ satisfying $\iota(B)\theta=1$.  Then $L(B)\theta=0$
and $B$ is $\F_\theta$-foliate.  Let $\tilde{B}$ be the lift of $B$ to
a presentation $(\tilde{M},\tilde{\omega},\Gamma,\chi,\tilde{f})$.
Then $L(\tilde{B})\tilde{f}=1$.
\end{lemma}

\begin{proof}
It follows from $\iota(B)\theta=1$ that
$L(B)\theta=[\iota(B),d]\theta=d\iota(B)\theta=0$.  Let
$X\in\X(\F_\theta)$, i.e.\ $\iota(X)\theta=0$.  Then
\[\iota([B,X])\theta=L(B)\iota(X)\theta-\iota(X)L(B)\theta=0,\]
so $[B,X]\in\X(\F_\theta)$.  Hence $B$ is foliate.  It follows from
$\iota(B)\theta=1$ that $\iota(\tilde{B})d\tilde{f}=1$ and hence that
$L(\tilde{B})\tilde{f}=1$.
\end{proof}

The vector field $B$ in this lemma can be averaged over $K$ to be made
$K$-invariant.  If $B$ is complete, its flow maps leaves of
$\F_\theta$ to leaves of $\F_\theta$ in a $K$-equivariant fashion.
Lifting the flow of $B$ to the minimal presentation $\tilde{M}$ of $M$
induces a diffeomorphism $\tilde{M}\cong\R\times L$, where $L$ is any
leaf of $\F_\theta$, and gives the following version of the Reeb
stability theorem (see e.g.\ \cite[Ch.~2]{MM03}) for the Lee
foliation.

\begin{theorem}\label{theorem;lee-leaf}
Let $(M,\omega,\theta)$ be a connected strict conformal symplectic
manifold equipped with a Hamiltonian $K$-action with moment map
$\Phi$.  Let $(\tilde{M},\tilde\omega,\Gamma,\chi,\tilde{f})$ be the
minimal presentation of $(M,\omega,\theta)$.  Suppose also that the
Lee form $\theta$ vanishes nowhere and that there exists a complete
$K$-invariant vector field $B$ satisfying $\iota(B)\theta=1$.
\begin{enumerate}
\item\label{item;lift-foliation}
Let $\tilde{\F}_\theta$ be the lift of the Lee foliation $\F_\theta$
to $\tilde{M}$.  The leaves of $\tilde{\F}_\theta$ are the level sets
of the potential $\tilde{f}\colon\tilde{M}\to\R$.  For each leaf
$\tilde{L}$ of $\tilde{\F}_\theta$ the covering map
$p\colon\tilde{M}\to M$ restricts to a diffeomorphism $\tilde{L}\to
L$, where $L$ is a leaf of $\F_\theta$.
\item\label{item;foliations}
Let $\tilde{B}$ be the lift of $B$ to $\tilde{M}$ and let
$F\colon\R\times\tilde{M}\to\tilde{M}$ be the flow of $\tilde{B}$.
Let $\tilde{L}_0=\tilde{f}^{-1}(0)$, $L_0=p(\tilde{L}_0)$, and
$q\colon L_0\to\tilde{L}_0$ the inverse of $p\colon\tilde{L}_0\to
L_0$.  Define $F_0\colon\R\times L_0\to\tilde{M}$ by
$F_0(t,x)=F(0,q(x))$.  Equip $\R\times L_0$ with the foliation given
by the fibres of the projection $\pr_1\colon\R\times L_0\to\R$.  Then
$F_0$ is a $K$-equivariant foliate diffeomorphism.  The potential
$\tilde{f}\colon\tilde{M}\to\R$ descends to a continuous map $f\colon
M\cong\tilde{M}/\Gamma\to\R/\Gamma$, whose fibres are the leaves of
$\F_\theta$.  In other words, we have the following commutative
diagram, in which the top row consists of foliated manifolds and
foliate maps and the bottom row consists of the respective leaf
spaces:
\begin{equation}\label{equation;foliations}
\begin{tikzcd}[sep=large]
\R\times L_0\ar[r,"F_0","\cong"']\ar[d,"\pr_1"']&
\tilde{M}\ar[d,"\tilde{f}"']\ar[r,"/\Gamma"]&M\ar[d,"f"]
\\
\R\ar[r,"="]&\R\ar[r,"/\Gamma"]&\R/\Gamma
\end{tikzcd}
\end{equation}
The formula $\gamma*x=p\bigl(F(-\chi(\gamma),\gamma\cdot q(x))\bigr)$
for $\gamma\in\Gamma$ and $x\in L_0$ defines an action of $\Gamma$ on
the leaf $L_0$.  The map $F_0$ is equivariant with respect to the
diagonal $\Gamma$-action, so $M$ is diffeomorphic to the quotient
$(\R\times L_0)/\Gamma$.

\item\label{item;closed-dense}
If $\rank(M)=1$ every leaf of $\F_\theta$ is closed.  If
$\rank(M)\ge2$ every leaf of $\F_\theta$ is dense.
\item\label{item;leaf-convex}
Suppose the $K$-action on $M$ is of Lee type with central Lee element
$\zeta\in\lie{z}(\lie{k})$.  Let $L$ be any leaf of $\F_\theta$ and
let $\Delta(L)=\Phi(L)\cap C$ be the moment body of the presymplectic
Hamiltonian $K$-manifold $L$.  Then $\Delta(L)=\Delta(M)$.  In
particular $\Delta(L)$ is independent of the leaf $L$ and, if $\Phi$
is proper, is a convex polyhedral set.
\end{enumerate}
\end{theorem}

\begin{proof}
\eqref{item;lift-foliation}--\eqref{item;foliations}~Since
$p^*\theta=d\tilde{f}$, the leaves of $\tilde{\F}_\theta$ are the
connected components of the level surfaces of $\tilde{f}$.  The
formula $\gamma^*f=f+\chi(\gamma)$ and the injectivity of
$\chi\colon\Gamma\to\R$ show that the restriction of $p$ to each level
surface of $\tilde{f}$ is injective.  Thus $p$ maps each leaf of
$\tilde{\F}_\theta$ diffeomorphically to a leaf of $\F_\theta$.  We
must show that the level sets of $\tilde{f}$ are connected.  Since $B$
is complete, so is $\tilde{B}$.  It follows from
$L(\tilde{B})\tilde{f}=1$ (Lemma~\ref{lemma;lee-transverse}) that
\[
\frac{d}{dt}\tilde{f}(F(t,\tilde{x}))=1,\qquad
\tilde{f}(F(0,\tilde{x}))=\tilde{f}(\tilde{x})
\]
for all $\tilde{x}\in\tilde{M}$, whence
$\tilde{f}(F(t,\tilde{x}))=t+\tilde{f}(\tilde{x})$ for all $t\in\R$.
This shows that $F_0$ is a diffeomorphism and that the left square in
the diagram~\eqref{equation;foliations} commutes.  Since $\tilde{B}$
is $\tilde{\F}_\theta$-foliate and $K$-invariant, $F_0$ is
$\tilde{\F}_\theta$-foliate and $K$-equivariant.  By assumption
$\tilde{M}$ is connected, and therefore so is $\tilde{L}_0$ and every
other level set of $\tilde{f}$.  The potential $\tilde{f}$ is
$\Gamma$-equivariant and therefore descends to a continuous $f\colon
M\to\R/\Gamma$, which makes the square on the right
in~\eqref{equation;foliations} commute.  The flow $F$ is
$\Gamma$-equivariant, so it follows from the flow law that
\[
F(\chi(\gamma),\gamma*\tilde{x})=
F(\chi(\gamma),F(-\chi(\gamma),\gamma\cdot\tilde{x}))=\gamma\cdot
F(\chi(\gamma),F(-\chi(\gamma),\tilde{x}))=\gamma\cdot F(0,\tilde{x})
\]
for all $\gamma\in\Gamma$ and $\tilde{x}\in\tilde{L}_0$, which proves
the last two assertions of~\eqref{item;foliations}.

\eqref{item;closed-dense}~The leaves of $\F_\theta$ are the fibres of
$f\colon M\to\R/\Gamma$.  If $\rank(M)=1$ we have
$\R/\Gamma\cong\bb{S}^1$ and $f$ is a smooth locally trivial fibre
bundle, so the fibres are closed.  Now let $\rank(M)\ge2$.  For
$t\in\R$ let $\tilde{L}_t$ be the leaf $\tilde{f}^{-1}(t)$ of
$\tilde{\F}_\theta$ and let $L_t$ be the leaf $p(\tilde{L}_t)$ of
$\F_\theta$.  The inverse image of $L_t$ in $\R\times L_0$ is
\[
(p\circ F_0)^{-1}(L_t)= \Gamma\cdot(\{t\}\times
L_0)=(t+\chi(\Gamma))\times L_0,
\]
which is dense in $\R\times L_0$.  Hence $L$ is dense in $M$.

\eqref{item;leaf-convex}~The leaf $\tilde{L}_t$ is coisotropic in
$\tilde{M}$, hence has presymplectic form
$\tilde{\omega}_t=\tilde{\omega}|_{\tilde{L}_t}$, and $L_t$ has
presymplectic form $\omega_t=\omega|_{L_t}$.  We have
$\tilde{\omega}_t=e^tp^*\omega_t$, so
$\Delta(\tilde{L}_t)=e^t\Delta(L_t)$.  On the other hand
$\Phi^\zeta=1$ on $M$, so $\tilde{\Phi}^\zeta=e^{\tilde{f}}$ on
$\tilde{M}$, so $\tilde{L}_t=(\tilde{\Phi}^\zeta)^{-1}(e^t)$.
Therefore
\[
e^t\Delta(L_t)=\Delta(\tilde{L}_t)=\Delta(\tilde{M})\cap
e^t\H_\zeta=e^t\bigl(\Delta(\tilde{M})\cap\H_\zeta\bigr)=
e^t\Delta(M),
\]
where we used Theorems~\ref{theorem;rank-1} and~\ref{theorem;rank-2}.
Hence $\Delta(L_t)=\Delta(M)$.  The last assertion follows from
Theorem~\ref{theorem;main}.
\end{proof}

\begin{remarks}
\begin{numerate}
\item\label{item;action}
If $\rank(M)=1$ the fact that $M\cong(\R\times L_0)/\Gamma$ shows that
$f\colon M\to\R/\Gamma$ is the fibre bundle associated to the
principal $\Gamma$-bundle $\R\to\R/\Gamma\cong\bb{S}^1$ with fibre the
$\Gamma$-space $L_0$.  The trajectories of $B$ define a flat Ehresmann
connection on the bundle $M$.
\item\label{item;presymplectic-convex}
For closed leaves $L$ the fact that $\Delta(L)$ is convex polyhedral
follows also from the presymplectic convexity theorem of~\cite{LS17}.
\item\label{item;morse-bott}
If the Lee form has zeroes, some of the leaves of $\F_\theta$ may be
singular and part~\eqref{item;foliations} of the theorem may fail.
Most of the other statements are however still true.  Since the
components of the symplectic moment map $\tilde{\Phi}$ are Morse-Bott
functions, the equality $e^{\tilde{f}}=\tilde{\Phi}^\zeta$ implies
that the potential $\tilde{f}$ is also Morse-Bott, so the Lee form
$\theta$ has Morse-Bott singularities, and all its Morse indices are
even.  Work of Imanishi~\cite{Im79} and Farber~\cite[\S\,9.1]{Fa04}
shows that the leaves of $\F_\theta$ are, locally, the level sets of
$\tilde{f}$, and that the closed-dense
dichotomy~\eqref{item;closed-dense} still holds.
Part~\eqref{item;leaf-convex} is still true for the regular leaves and
presumably also for the singular leaves.
\end{numerate}
\end{remarks}

%%%%%%%%%%%%%%%%%%%%%%%%%%%%%%%%%%%%%%%%%%%%%%%%%%%%%%%%%%%%%%%%%%%%%%%%
\section{Convexity for conformal symplectic manifolds of the first kind}
\label{section;kind}
%%%%%%%%%%%%%%%%%%%%%%%%%%%%%%%%%%%%%%%%%%%%%%%%%%%%%%%%%%%%%%%%%%%%%%%%

In this section we explore yet another relationship between conformal
symplectic geometry and presymplectic geometry, which exists only for
conformal symplectic manifolds of the first kind (to be defined
below).  The upshot is a structure theorem,
Theorem~\ref{theorem;structure}, which expresses every such manifold
as a generalized type of contact mapping torus, and a strong version
of the conformal symplectic convexity theorem,
Theorem~\ref{theorem;conformal-first}, which is based on the
presymplectic convexity theorem of~\cite{LS17}.  One advantage of
these results is that they tell us when the moment polyhedron is
semirational.  Another advantage is a geometric explanation for the
stability of the polytopes $\Delta(L)$ for leaves $L$ of the Lee
foliation: in Theorem~\ref{theorem;lee-leaf} the fact that all leaves
have the same moment polytope seems almost accidental, but for
manifolds of the first kind these leaves are actually contact
manifolds and there is a flow that maps them contactomorphically one
onto the other, so that their moment polytopes are manifestly the
same.

These facts rest on Vaisman's observation that on a conformal
symplectic manifold of the first kind there exists a natural
associated presymplectic form.  It turns out that a group action which
is Hamiltonian with respect to the conformal symplectic structure is
also Hamiltonian with respect to the presymplectic structure with the
same moment map.  We will see that the Lee type condition corresponds
to the action being clean with respect to the null foliation of the
presymplectic form.

Let $M$ be a connected manifold.  A \emph{conformal symplectic
  structure of the first kind} on $M$ is a triple $(\omega,\theta,B)$,
where $(\omega,\theta)$ is a conformal symplectic structure and $B$ is
a vector field in $\X_\omega(M)$ such that the function
$\iota(B)\theta$ (which is constant because $L(B)\theta=0$) is equal
to~$1$.  We call such a vector field $B$ an \emph{anti-Lee vector
  field}.  We say a conformal symplectic manifold of the first kind
$(M,\omega,\theta,B)$ is \emph{complete} if its anti-Lee vector field
$B$ is complete.  On a conformal symplectic manifold of the first kind
the Lee form $\theta$ is nowhere $0$, so the Lee foliation $\F_\theta$
is regular.  The following result of Vaisman states that conformal
symplectic structures of the first kind are $d_\theta$-exact,
i.e.\ $\omega=d_\theta\alpha$, and that the form $d\alpha$ is
presymplectic of corank~$2$.

\begin{proposition}[{\cite[Proposition~2.2]{Va85}}]
\label{first-kind}
Let $M$ be a connected $2n$-dimen\-sional manifold.
\begin{enumerate}
\item\label{item;con-pre}
Given a conformal symplectic structure of the first kind
$(\omega,\theta,B)$ on $M$, the $1$-form $\alpha=\iota(B)\omega$
satisfies
\begin{equation}\label{first}
\rank(d\alpha)<2n,\qquad\text{$\theta\wedge\alpha\wedge(d\alpha)^{n-1}$
  is a volume form},
\end{equation}
and $\omega=d_\theta\alpha$.  The Lee vector field $A$ of
$(\omega,\theta)$ commutes with $B$ and is everywhere linearly
independent of $B$.  The form $d\alpha$ is presymplectic with null
foliation $\F_{d\alpha}$ spanned by $A$ and $B$.
\item
Conversely, given a $1$-form $\alpha$ and a closed $1$-form $\theta$
on $M$ satisfying \eqref{first}, the form $\omega=d_\theta\alpha$ is
conformal symplectic of the first kind with Lee form~$\theta$ and
anti-Lee vector field $B$ determined by $\alpha=\iota(B)\omega$.
\end{enumerate}
\end{proposition}

%%%%%%%%%%%%%%%%%%%%%%%%%%%%%%%%%%%%%%%%%%%%%%%%%%%%%%%%%%%%%%%%%%%%%%%%
\subsection*{Contact mapping tori of higher rank}
%%%%%%%%%%%%%%%%%%%%%%%%%%%%%%%%%%%%%%%%%%%%%%%%%%%%%%%%%%%%%%%%%%%%%%%%

Here is a general method for producing complete conformal symplectic
manifolds of the first kind.  Start with a contact form $\alpha$ on a
manifold $Q$.  The \emph{conformal symplectic cylinder} over $Q$ is
the product $\tilde{M}=\R\times Q$ equipped with the global conformal
symplectic structure $(d\alpha+dt\wedge\alpha,dt)$.  The Lee vector
field of the conformal symplectic cylinder is $\tilde{A}=-R$, where
$R$ is the Reeb vector field of $Q$.  The vector field
$\tilde{B}=\partial/\partial t$ is anti-Lee.  Fix $k$ commuting strict
contactomorphisms, i.e.\ diffeomorphisms $g_i\colon Q\to Q$ satisfying
$g_i^*\alpha=\alpha$ and $g_i\circ g_j=g_j\circ g_i$ for $i$,~$j=1$,
$2,\dots$,~$k$, and $k$ real numbers $c_1$, $c_2,\dots$,~$c_k$.  The
maps $\gamma_i\colon\tilde{M}\to\tilde{M}$ defined by
$\gamma_i(t,x)=(t+c_i,g_i(x))$ are commuting automorphisms of the
conformal symplectic cylinder, which generate an action of
$\Gamma=\Z^k$.  We assume this $\Gamma$-action to be free and proper.
(For $k=1$ this is automatically the case if $c_1\ne0$.  For $k\ge2$
the action is free if $c_1$, $c_2,\dots$,~$c_k$ are independent over
$\Q$, and can be proper only if $Q$ is noncompact.  See
Section~\ref{section;examples} for what to do when the $\Gamma$-action
is not free and proper.)  The \emph{contact mapping torus} of $Q$ with
respect to the $g_i$ and the $c_i$ is the quotient manifold
\[M=M_{(g_1,g_2,\dots,g_k;c_1,c_2,\dots,c_k)}=\tilde{M}/\Gamma\]
equipped with the strict conformal symplectic structure
$(\omega,\theta)$ induced by $(d\alpha+dt\wedge\alpha,dt)$.
Proposition~\ref{first-kind} shows that $(\omega,\theta)$ is of the
first kind.  The anti-Lee vector field $B$ of the contact mapping
torus $M$ is induced by $\tilde{B}$, so $M$ is complete.  The
projection $\tilde{M}=\R\times Q\to\R$ induces a fibration of $M$ over
the quotient $\R/\Gamma$, which is a genuine circle if $k=1$ and a
``stacky circle'' if $k\ge2$.

Let $\tilde{\omega}=e^{\tilde{f}}(d\alpha+dt\wedge\alpha)$, where
$\tilde{f}(t,x)=t$, and let $\chi\colon\Gamma\to\R$ be the inclusion.
Then $(\tilde{M},\tilde{\omega})$ is a symplectic manifold, namely the
symplectization of the contact manifold $Q$, and the tuple
$(\tilde{M},\tilde\omega,\Gamma,\chi,\tilde{f})$ is a presentation of
the contact mapping torus $M$.  If the $c_j$ are independent over
$\Q$, the homomorphism $\chi$ is injective, so the presentation is
minimal and the rank of $M$ is equal to $k$.

The following Boothby-Wang type structure theorem, versions of which
were stated in the rank $1$ case by
Vaisman~\cite[Proposition~3.3]{Va85} and
Bazzoni-Marrero~\cite[Theorem~4.6]{BM18}, says that every connected
complete strict conformal symplectic manifold of the first kind is a
contact mapping torus.

\begin{theorem}\label{theorem;structure}
Let $(M,\omega,\theta,B)$ be a connected complete strict conformal
symplectic manifold of the first kind.  Let
$(\tilde{M},\tilde\omega,\Gamma,\chi,\tilde{f})$ be the minimal
presentation of $(M,\omega,\theta)$.
\begin{enumerate}
\item\label{item;contact}
Let $\tilde{B}=p^*B$ and
$\tilde\alpha=e^{\tilde{f}}p^*\alpha=\iota(\tilde{B})\tilde{\omega}$.
For $t\in\R$ let
\[
\tilde{Q}_t=\tilde{f}^{-1}(t),\quad
\tilde{\alpha}_t=\tilde{\alpha}|_{\tilde{Q}_t},\quad
Q_t=p(\tilde{Q}_t),\quad\alpha_t=\alpha|_{Q_t}.
\]
Then $(\tilde{Q}_t,\tilde{\alpha}_t)$ and $(Q_t,\alpha_t)$ are contact
manifolds for all $t$.  The Reeb vector field on $Q_t$ is equal to
$-A$, where $A$ is the Lee vector field of $M$.  The flow of $B$
induces strict contactomorphisms $Q_0\to Q_t$ for all $t$.
\item\label{item;symplectization}
We have $F_0^*\tilde{\omega}=e^t(d\alpha_0+dt\wedge\alpha_0)$, where
$F_0\colon\R\times Q_0\to\tilde{M}$ is the diffeomorphism of
Theorem~\ref{theorem;lee-leaf}\eqref{item;foliations}.  It follows
that $(\tilde{M},\tilde{\omega})$ is the symplectization of
$(Q_0,\alpha_0)\cong(\tilde{Q}_0,\tilde{\alpha}_0)$.
\item\label{item;mapping-torus}
We have $(p\circ F_0)^*\omega=d\alpha_0+dt\wedge\alpha_0$.  It follows
that $(\tilde{M},p^*\omega)\cong(\R\times
Q_0,d\alpha_0+dt\wedge\alpha_0,dt)$ is the conformal symplectic
cylinder over $Q_0$ and that its quotient by the $\Gamma$-action
$(M,\omega,\theta)$ is a contact mapping torus.
\end{enumerate}
\end{theorem}

\begin{proof}
\eqref{item;contact}~Let $\dim(M)=2n$.  It follows from
Proposition~\ref{first-kind} that
\[
\alpha\wedge(d\alpha)^{n-1}=
\alpha\wedge(\omega-\theta\wedge\alpha)^{n-1}=\alpha\wedge\omega^{n-1}.
\]
Now $\theta\wedge\alpha\wedge\omega^{n-1}$ is a volume form on $M$, so
$\alpha\wedge(d\alpha)^{n-1}$ restricts to a volume form on $Q_t$, so
$\alpha_t$ is a contact form.  We have
$\tilde{\alpha}_t=e^tp^*\alpha_t$, so $\tilde{\alpha}_t$ is contact as
well.  Also
\begin{align*}
\iota(A)\alpha&=-\iota(B)\iota(A)\omega=-\iota(B)\theta=-1,\\
\iota(A)d\alpha&=\iota(A)(\omega-\theta\wedge\alpha)=
\theta+(\iota(A)\alpha)\theta=0,
\end{align*}
so $-A$ is the Reeb vector field of $\alpha$.
Theorem~\ref{theorem;lee-leaf}\eqref{item;foliations} says that the
time $t$ flow of $B$ maps $Q_0$ to $Q_t$.  From $L(B)\omega=0$ we get
$L(B)\alpha=L(B)\iota(B)\omega=-\iota(B)L(B)\omega=0$, so the flow of
$B$ preserves $\alpha$.

\eqref{item;symplectization}--\eqref{item;mapping-torus}~Using
$\omega=d_\theta\omega$ (Proposition~\ref{first-kind}),
$F_0^*p^*\alpha=\alpha_0$, and $F_0^*\tilde{f}=t$ gives
\[
F_0^*p^*\omega=F_0^*p^*(d\alpha+\theta\wedge\alpha)=
d\alpha_0+dt\wedge\alpha_0.
\]
Hence also $F_0^*\tilde{\omega}=F_0^*(e^{\tilde{f}}p^*\omega)=
e^t(d\alpha_0+dt\wedge\alpha_0)$.
\end{proof}

In the presence of a Hamiltonian action we have the following
additional information.

\begin{proposition}\label{first-kind-ham}
Let $(M,\omega,\theta,B)$ be a connected strict conformal symplectic
manifold of the first kind which is equipped with a Hamiltonian
$K$-action with moment map $\Phi\colon M\to\lie{k}^*$.
\begin{enumerate}
\item\label{item;invariant}
We may assume the anti-Lee vector field $B$ to be $K$-invariant.  For
such a choice of $B$ we have $L(B)\Phi^\xi=0$ for all $\xi\in\lie{k}$.
\item\label{first-pre}
Suppose $B$ is $K$-invariant and let $\alpha=\iota(B)\omega$.  The
$K$-action is Hamiltonian with respect to the presymplectic form
$d\alpha$ and with the same moment map $\Phi$.
\item\label{lee-clean}
Suppose the $K$-action on $(M,\omega,\theta)$ is of Lee type and let
$\zeta\in\lie{z}(\lie{k})$ be a Lee element.  Then the $K$-action on
$(M,\F_{d\alpha})$ is clean with null ideal $\R\zeta\oplus\lie{n}$,
where $\lie{n}=\ker(\lie{k}\to\X(M))$ is the kernel of the
infinitesimal action.
\item\label{item;cone-cylinder}
Suppose $B$ is $K$-invariant and complete.  Let
$(\tilde{M},\tilde\omega,\Gamma,\chi,\tilde{f})$ be the minimal
presentation of $(M,\omega,\theta)$, and let $Q_0$ and
$F_0\colon\R\times Q_0\to\tilde{M}$ be as in
Theorem~\ref{theorem;structure}.  Let $\Psi\colon Q_0\to\lie{k}^*$ be
the contact moment map for the $K$-action on $Q_0$ defined by
$\Psi^\xi=-\iota(\xi_{Q_0})\alpha_0$ for $\xi\in\lie{k}$.  Then the
moment map for the $K$-action on the conformal symplectic cylinder
$(\tilde{M},p^*\omega)\cong(\R\times
Q_0,d\alpha_0+dt\wedge\alpha_0,dt)$ is given by
\[(p\circ F_0)^*\Phi(t,x)=\Psi(x),\]
and the moment map for the $K$-action on the symplectic cone
$(\tilde{M},\tilde{\omega})\cong\bigl(\R\times
Q_0,e^t(d\alpha_0+dt\wedge\alpha_0)\bigr)$ is given by
\[F_0^*\tilde{\Phi}(t,x)=e^t\Psi(x).\]
\end{enumerate}
\end{proposition}

\begin{proof}
\eqref{item;invariant}~Let $\bar{B}=\int_Kg_*B\,dg$.  Using the
$K$-invariance of $\omega$ and $\theta$
(Proposition~\ref{proposition;equivariant-moment}) we get
\[
\iota(\bar{B})\theta=\int_K\iota(g_*B)\theta\,dg
=\int_Kg^*(\iota(B)\theta)\,dg=\int_Kg^*(1)\,dg=1,
\]
and similarly $L(\bar{B})\omega=0$, so $\bar{B}$ is anti-Lee and
$K$-invariant.  Let $\xi\in\lie{k}$ and suppose $B$ is $K$-invariant.
Then $[\xi_M,B]=0$, so
$L_\theta(B)\iota(\xi_M)=\iota(\xi_M)L_\theta(B)$
by~\eqref{equation;cartan}.  Using
$d_\theta\Phi^\xi=\iota(\xi_M)\omega$ yields
\begin{align*}
d_\theta L_\theta(B)\Phi^\xi&=L_\theta(B) d_\theta\Phi^\xi=
L_\theta(B)\iota(\xi_M)\omega=\iota(\xi_M)L_\theta(B)\omega\\
&=\iota(\xi_M)(L(B)\omega+\theta(B)\omega)=\iota(\xi_M)\omega=
d_\theta\Phi^\xi,
\end{align*}
so $L_\theta(B)\Phi^\xi=\Phi^\xi$ by injectivity of $d_\theta$
(Proposition~\ref{proposition;strict}).  Hence $L(B)\Phi^\xi=0$.

\eqref{first-pre}~We must show that $d\Phi^\xi=\iota(\xi_M)d\alpha$
for all $\xi\in\lie{k}$.
Proposition~\ref{first-kind}\eqref{item;con-pre} gives
$\omega=d_\theta\alpha$.  Using $d_\theta\Phi^\xi=\iota(\xi_M)\omega$
and $L_\theta(\xi_M)=[d_\theta,\iota(\xi_M)]$
(see~\eqref{equation;cartan}) we get
\[
d_\theta\Phi^\xi=\iota(\xi_M)d_\theta\alpha=
L_\theta(\xi_M)\alpha-d_\theta\iota(\xi_M)\alpha=
L(\xi_M)\alpha+\theta(\xi_M)\cdot\alpha-d_\theta\iota(\xi_M)\alpha.
\]
Since $B$ and $\omega$ are invariant, $\alpha=\iota(B)\omega$ is
invariant, so $L(\xi_M)\alpha=0$.  Since $\iota(\xi_M)\theta=0$
(Proposition~\ref{proposition;equivariant-moment}) and
$L(\xi_M)\alpha=0$, this gives
$d_\theta\Phi^\xi=-d_\theta(\iota(\xi_M)\alpha)$.  By
Proposition~\ref{proposition;strict},
$d_\theta\colon\ca{C}^{\infty}(M)\to\Omega^1(M)$ is injective, so
$\Phi^\xi=-\iota(\xi_M)\alpha$ and thus
$d\Phi^\xi=-d\iota(\xi_M)\alpha=\iota(\xi_M)d\alpha$.

\eqref{lee-clean}~Let $\xi\in\lie{k}$ and suppose $\xi_M$ is tangent
to the null foliation $\F_{d\alpha}$, which according to
Proposition~\ref{first-kind} is spanned by $A=\zeta_M$ and $B$.  Then
$\xi_M=\lambda\zeta_M+\mu B$ for some $\lambda$ and $\mu$.  Using
$\iota(\xi_M)\theta=\iota(\zeta_M)\theta=0$ and $\iota(B)\theta=1$ we
obtain $\xi_M=\lambda\zeta_M$, so $\xi\in\R\zeta\oplus\lie{n}$.  This
shows that the null ideal is equal to $\R\zeta\oplus\lie{n}$ and that
\[T_m(K\cdot m)\cap T_m\F=\R\zeta_M(m)=T_m(N\cdot m)\]
for all $m\in M$.

\eqref{item;cone-cylinder}~Let $\xi\in\lie{k}$.  Since $\alpha_0$ and
$\xi_{Q_0}$ are $\Gamma$-invariant, the function $\Psi^\xi$ is
$\Gamma$-invariant and therefore descends to a function $\psi^\xi$ on
$M$.  It follows from~\eqref{item;invariant} that
\[
\psi^\xi=-\iota(\xi_M)\alpha=-\iota(\xi_M)\iota(B)\omega=
\iota(B)\iota(\xi_M)\omega=\iota(B)d_\theta\Phi^\xi=L_\theta(B)\Phi^\xi=
\Phi^\xi,
\]
which proves the first statement.  The second statement follows
immediately from the first.
\end{proof}

Combining Theorem~\ref{theorem;structure} and
Proposition~\ref{first-kind-ham} with the presymplectic convexity
theorem~\cite[Theorem 2.2]{LS17} yields the following result.

\begin{theorem}\label{theorem;conformal-first}
Let $(M,\omega,\theta,B)$ be a connected strict conformal symplectic
manifold of the first kind equipped with a Hamiltonian $K$-action with
moment map $\Phi\colon M\to\lie{k}^*$ and with an anti-Lee vector
field $B$ that is $K$-invariant and complete.  Let $T$ be a maximal
torus of $K$ and $\lie{t}=\Lie(T)$.  Choose a closed Weyl chamber $C$
in $\lie{t}^*$ and define $\Delta(M)=\Phi(M)\cap C$.  Suppose the
$K$-action is of Lee type with central Lee element $\zeta\in\lie{k}$.
\begin{enumerate}
\item
All leaves $Q$ of the Lee foliation $\F_\theta$ of $M$ are
equivariantly contactomorphic to one another.  We have
$\Delta(Q)=\Delta(M)$ for every $Q$, and
$\Delta(\tilde{M})=\R_{>0}\cdot\Delta(M)$ and
$\Delta(\tilde{M})=\Delta(M)\cap\H_\zeta$ for every presentation
$\tilde{M}$ of $M$.
\item
Suppose $\Phi$ is proper.  Then $\Phi\colon M\to\Phi(M)$ is an open
map with connected fibres; $\Delta(M)$ is a closed convex polyhedral
set; $\Delta(\tilde{M})\cup\{0\}$ is a convex polyhedral cone; and
$\Delta(M)$ is semirational if and only if the line $\R\zeta$ in
$\lie{t}$ is rational.
\end{enumerate}
\end{theorem}

%%%%%%%%%%%%%%%%%%%%%%%%%%%%%%%%%%%%%%%%%%%%%%%%%%%%%%%%%%%%%%%%%%%%%%%%
\section{Examples}\label{section;examples}
%%%%%%%%%%%%%%%%%%%%%%%%%%%%%%%%%%%%%%%%%%%%%%%%%%%%%%%%%%%%%%%%%%%%%%%%

%%%%%%%%%%%%%%%%%%%%%%%%%%%%%%%%%%%%%%%%%%%%%%%%%%%%%%%%%%%%%%%%%%%%%%%%
\subsection*{Cotangent bundles}
%%%%%%%%%%%%%%%%%%%%%%%%%%%%%%%%%%%%%%%%%%%%%%%%%%%%%%%%%%%%%%%%%%%%%%%%

Here is an example of a Hamiltonian action on a strict conformal
symplectic manifold which is not of Lee type, but where the conclusion
of the convexity theorem nevertheless holds.  Let $Q$ be a connected
$K$-manifold with $H_\bas^1(Q)\ne0$.  Let $M=T^*Q$ be its cotangent
bundle and $\pi\colon M\to Q$ the projection.  The $K$-action on $Q$
lifts naturally to a $K$-action on $M$.  Let $\alpha\in\Omega^1(M)$ be
the tautological $1$-form; then $\omega_0=d\alpha$ is the standard
symplectic form on $M$.  We modify $\omega_0$ as follows: we choose a
closed but not exact basic $1$-form $\theta_Q\in\Omega_\bas^1(Q)$, and
we put $\theta=\pi^*\theta_Q$ and $\omega=d_\theta\alpha$.  Then
$(\omega,\theta)$ is a $K$-invariant strict conformal symplectic
structure on $M$.  For $\xi\in\lie{k}$ define a function $\Phi^\xi$ on
$M$ by $\Phi^\xi=-\iota(\xi_M)\alpha$; then
\[
d_\theta\Phi^\xi=-L_\theta(\xi_M)\alpha+\iota(\xi_M)d_\theta\alpha=
\iota(\xi_M)\omega,
\]
because $\theta$ is basic.  So the action is Hamiltonian with moment
map $\Phi$.  However, the action is not of Lee type.  If there existed
a Lee element $\zeta\in\lie{k}$, we would have
$1=\Phi^\zeta=-\iota(\zeta_M)\alpha$, i.e.\ $\beta(\zeta_M)=-1$ for
all $q\in Q$ and $\beta\in T_q^*Q$.  Taking $\beta=(\theta_Q)_q$ and
using that $\theta_Q$ is basic gives $0=(\theta_Q)_q(\zeta_M)=-1$, a
contradiction.  Nevertheless the intersection $\Phi(M)\cap C$ is
convex, indeed a rational convex polyhedral cone.  This follows from
the fact that $\Phi$ is also the moment map for the standard
symplectic form $\omega_0$ and from the symplectic convexity theorem
for cotangent bundles,~\cite[Theorem~7.6]{Sj98}.

This example prompts the question whether the Lee type hypothesis of
Theorem~\ref{theorem;main} can be weakened.  Let $(M,\omega,\theta)$
be a strict conformal symplectic manifold equipped with a Hamiltonian
$K$-action.  Let $A$ be the Lee vector field and let $\F_A$ be the
(possibly singular) $1$-dimensional foliation of $M$ generated by $A$.
Define
\[
\lie{n}_A= \{\,\xi\in\lie{k}\mid\text{$\xi_M$ is tangent to
  $\F_A$}\,\};
\]
then $\lie{n}_A$ is an ideal of $\lie{k}$.  Let $N_A$ be the normal
subgroup of $K$ with Lie algebra $\lie{n}_A$.  Then $N_A\cdot x\subset
K\cdot x\cap\F_A(x)$ for all $x\in M$, where $\F_A(x)$ denotes the
$\F_A$-leaf of $x$, and hence $T_x(N_A\cdot x)\subset T_x(K\cdot
x)\cap T_x\F_A$.  If the action is of Lee type with Lee element
$\zeta$ we have $\lie{n}_A=\lie{n}+\R\zeta$, where $\lie{n}$ is the
kernel of the infinitesimal action $\lie{k}\to\X(M)$, and therefore
\begin{equation}\label{equation;clean}
T_x(N_A\cdot x)=T_x(K\cdot x)\cap T_x\F_A
\end{equation}
for all $x$.  In the cotangent example above the Lee vector field is
tangent to the fibres of $\pi\colon M\to Q$, so we have $T_x(K\cdot
x)\cap T_x\F_A=0$ for all $x$, and therefore~\eqref{equation;clean} is
satisfied as well.

\begin{problem}\label{problem;lee}
Suppose the $K$-action on $M$ satisfies the ``cleanness''
condition~\eqref{equation;clean} and that the moment map $\Phi\colon
M\to\lie{k}^*$ is proper.  Do the conclusions of the convexity
theorem, Theorem~\ref{theorem;main}, hold for $M$?
\end{problem}

%%%%%%%%%%%%%%%%%%%%%%%%%%%%%%%%%%%%%%%%%%%%%%%%%%%%%%%%%%%%%%%%%%%%%%%%
\subsection*{Contact mapping tori of rank $1$}
%%%%%%%%%%%%%%%%%%%%%%%%%%%%%%%%%%%%%%%%%%%%%%%%%%%%%%%%%%%%%%%%%%%%%%%%

%%%%%%%%%%%%%%%%%%%%%%%%%%%%%%%%%%%%%%%%%%%%%%%%%%%%%%%%%%%%%%%%%%%%%%%%
\subsubsection*{Contact ellipsoids}
%%%%%%%%%%%%%%%%%%%%%%%%%%%%%%%%%%%%%%%%%%%%%%%%%%%%%%%%%%%%%%%%%%%%%%%%

Let $c>1$ and let $h\colon\C^n\to\C^n$ be the homothety $h(z)=cz$.
Then $h^*\tilde{\omega}=c^2\tilde{\omega}$, where
$\tilde{\omega}=\frac{i}2\sum dz_j\wedge d\bar{z}_j$ is the standard
symplectic form.  Choose numbers
$\zeta_1\ge\zeta_2\ge\cdots\ge\zeta_n>0$ and let
$\tilde{f}(z)=\log\bigl(\frac12\sum_j\zeta_j\abs{z_j}^2\bigr)$.  Then
$\tilde{f}(h(z))=\tilde{f}(z)+2\log c$.  The pair
$(e^{-\tilde{f}}\tilde\omega,d\tilde{f})$ defines a global conformal
symplectic structure on $\tilde{M}=\C^n\setminus\{0\}$, which is
invariant under the cyclic group $\Gamma$ generated by $h$.  This
group acts properly and freely on $\tilde{M}$ because $c>1$, so the
quotient $M=\tilde{M}/\Gamma$ is a conformal symplectic manifold
called a \emph{Hopf manifold}.  As a conformal symplectic manifold $M$
is the trivial contact mapping torus obtained from the contact
ellipsoid
\[
Q=Q_\zeta=\tilde{f}^{-1}(0)=
\biggl\{z\in\C^n\biggm|\frac12\sum_j\zeta_j\abs{z_j}^2=1\biggr\}
\]
and the identity automorphism $Q\to Q$.  As a manifold
\[M=\bb{S}^1\times\bb{S}^{2n-1}\]
and therefore its second Betti number vanishes.  Hence it does not
admit any symplectic, let alone K\"ahler, structures, although it
admits many complex structures.  The minimal presentation of $M$ is
$(\tilde{M},\tilde{\omega},\Gamma,\chi,\tilde{f})$, where
$\chi\colon\Gamma\to\R$ is defined by $\chi(h^n)=2n\log c$.

We equip $\lie{u}(n)=\Lie(\U(n))$, the Lie algebra of anti-Hermitian
matrices, with the trace form $\inner{\xi,\eta}=-\tr(\xi\eta)$, and
use this to identify $\lie{u}(n)^*\cong\lie{u}(n)$.  The standard
action of the unitary group $\U(n)$ on $\C^n$ commutes with the
$\Gamma$-action and is Hamiltonian relative to $\tilde{\omega}$ with
moment map given by $\tilde{\Phi}(z)=\frac{i}2zz^*$.  Here we think of
$z$ as a column vector and denote by $z^*$ its conjugate transpose.
Let $K$ be a closed subgroup of $\U(n)$ whose Lie algebra $\lie{k}$
contains $\zeta=i\diag(\zeta_1,\zeta_2,\dots,\zeta_n)$ as a central
element.  (If the $\zeta_j$ are pairwise distinct, this forces $K$ to
be the diagonal maximal torus of $\U(n)$.)  The $K$-action is
Hamiltonian with moment map given by
$\tilde{\Phi}^\xi(z)=-\frac{i}2z^*\xi z$ for $\xi\in\lie{k}$ and
$z\in\C^n$.  We have $e^{\tilde{f}}=\tilde{\Phi}^\zeta$, so the
induced $K$-action on $M$ is Hamiltonian with moment map
\[
\Phi^\xi([z])=e^{-\tilde{f}(z)}\tilde{\Phi}^\xi(z)=\frac{z^*\xi
  z}{z^*\zeta z}
\]
and is of Lee type with central Lee element $\zeta$.  The moment map
$\tilde{\Phi}$ is proper (because $\tilde{f}$ is proper) and the
moment body $\Delta(\tilde{M})$ is the rational convex polyhedral cone
spanned by the highest weights of the irreducible $K$-modules
occurring in the symmetric algebra of $\C^n$
(\cite[Theorem~4.9]{Sj98}).  The moment body $\Delta(M)=\Delta(Q)$ is
the convex polytope $\Delta(\tilde{M})\cap\H_\zeta$.  It is
semirational if and only if the $\Q$-linear subspace of $\R$ spanned
by $\zeta_1$, $\zeta_2,\dots$,~$\zeta_n$ is one-dimensional.

%%%%%%%%%%%%%%%%%%%%%%%%%%%%%%%%%%%%%%%%%%%%%%%%%%%%%%%%%%%%%%%%%%%%%%%%
\subsubsection*{Contact hyperboloids}
%%%%%%%%%%%%%%%%%%%%%%%%%%%%%%%%%%%%%%%%%%%%%%%%%%%%%%%%%%%%%%%%%%%%%%%%

Fix numbers $\zeta_1\ge\zeta_2\ge\cdots\ge\zeta_n>0$ and let $\zeta$
be the diagonal matrix $\diag(\zeta_1,\zeta_2,\dots,\zeta_n)$.  Write
vectors in $\R^{2n}$ as pairs $(x,y)\in\R^n\times\R^n$ and define the
split quadratic form $q=q_\zeta$ on $\R^{2n}$ by $q(x,y)=x^\top\zeta
y=\sum_j\zeta_j x_jy_j$, where $x^\top$ denotes the transpose of the
column vector $x$.  Let
\[\tilde{M}=\{\,(x,y)\in\R^{2n}\mid q(x,y)>0\,\}\]
and define $\tilde{f}\colon\tilde{M}\to\R$ by $\tilde{f}(x,y)=\log
q(x,y)$.  Let $\tilde\omega=\sum_jdx_j\wedge dy_j$ be the standard
symplectic form.  Then
$\bigl(\tilde{M},e^{-\tilde{f}}\tilde{\omega},d\tilde{f},\tilde{B}\bigr)$
is a conformal symplectic manifold of the first kind with Lee and
anti-Lee vector fields given by
\[
\tilde{A}=\sum_j\zeta_j\biggl(x_j\frac{\partial}{\partial
  x_j}-y_j\frac{\partial}{\partial y_j}\biggr),\qquad
\tilde{B}=\frac12\sum_j\biggl(x_j\frac{\partial}{\partial
  x_j}+y_j\frac{\partial}{\partial y_j}\biggr).
\]
By Theorem~\ref{theorem;structure} $\tilde{M}$ is the conformal
symplectic cylinder over the contact ``hyperboloid''
\[Q=Q_\zeta=\tilde{f}^{-1}(0)=\{\,(x,y)\in\R^{2n}\mid q(x,y)=1\,\},\]
which is equipped with the contact form
$\iota(\tilde{B})\tilde{\omega}|_Q=
\frac12\sum_j(x_j\,dy_j-y_j\,dx_j)\big|_Q$.

Recall that $\lie{sp}(2n,\R)$, the Lie algebra of the symplectic group
$\Sp(2n,\R)$, consists of all block matrices
$\bigl(\begin{smallmatrix}a_{11}&a_{12}\\a_{21}&a_{22}\end{smallmatrix}\bigr)$
with blocks $a_{ij}\in\lie{gl}(n,\R)$ satisfying
\[a_{12}=a_{12}^\top,\quad a_{21}=a_{21}^\top,\quad a_{22}=-a_{11}^\top.\]
Similarly, the orthogonal group $\group{O}(\zeta)$ of the quadratic
form $q_\zeta$ has Lie algebra $\lie{o}(\zeta)$ consisting of all
block matrices
$\bigl(\begin{smallmatrix}a_{11}&a_{12}\\a_{21}&a_{22}\end{smallmatrix}\bigr)$
with blocks $a_{ij}\in\lie{gl}(n,\R)$ satisfying
\[
a_{12}=-\zeta^{-1}a_{12}^\top\zeta,\quad
a_{21}^\top=-\zeta^{-1}a_{21}\zeta,\quad
a_{22}=-\zeta^{-1}a_{11}^\top\zeta.
\]
Let $\lie{g}=\{\,\eta\in\lie{gl}(n,\R)\mid[\eta,\zeta]=0\,\}$ be the
centralizer of $\zeta\in\lie{gl}(n,\R)$ and $G$ the closed connected
subgroup of $\GL(n,\R)$ with Lie algebra $\lie{g}$.  The map
$\lie{gl}(n,\R)\to\lie{gl}(2n,\R)$ defined by $\eta\mapsto
\bigl(\begin{smallmatrix}\eta&0\\0&-\eta^\top\end{smallmatrix}\bigr)$
  is an isomorphism
\[
\begin{tikzcd}
\lie{g}\ar[r,"\cong"]&\lie{sp}(2n,\R)\cap\lie{o}(\zeta),
\end{tikzcd}
\]
which we will use to identify $\lie{g}$ with
$\lie{sp}(2n,\R)\cap\lie{o}(\zeta)$ and $G$ with the identity
component of $\Sp(2n,\R)\cap\group{O}(\zeta)$.  Thus $G$ acts on
$\R^{2n}=T^*\R^n$ by the lifted cotangent action
$g\cdot(x,y)=\bigl(gx,(g^{-1})^\top y\bigr)$, which preserves both the
symplectic form $\tilde{\omega}$ and the quadratic form $q_\zeta$.  A
split Cartan subalgebra of~$\lie{g}$ is
\[
\lie{h}=\{\,\eta=\diag(\eta_1,\eta_2,\dots,\eta_n)\mid
\eta_1,\eta_2,\dots,\eta_n\in\R\,\}.
\]
Observe that the Lee vector field is in $\lie{h}$, namely
$\tilde{A}=\zeta$.  Let $H$ be the Cartan subgroup of $G$ with Lie
algebra $\lie{h}$.  Choose
$\eta=\diag(\eta_1,\eta_2,\dots,\eta_n)\in\lie{h}$
%%
%% which is centralized by $G$ and such for each $j$ the numbers
%% $\zeta_j$ and $\eta_j$ are independent over $\Q$.
%%
and put
\[
h_0=\exp\zeta,\quad h_1=\exp\eta\in H.
\]
These elements generate subgroups $\inner{h_0}$ and $\inner{h_1}$ of
$H$, both of which act properly on $Q$, but the subgroup
$\inner{h_0,h_1}\cong\Z^2$ does not act properly.  However, choose
$c\ne0$ and define $\gamma_0$, $\gamma_1\colon\tilde{M}\to\tilde{M}$
by
\[
\gamma_0(x,y)=h_0\cdot(x,y),\qquad\gamma_1(x,y)=
e^{\frac12c}h_1\cdot(x,y).
\]
Then $\Gamma=\inner{\gamma_0,\gamma_1}\cong\Z^2$ acts properly and
freely on $\tilde{M}$.  We have
$\gamma_i^*\tilde{\omega}=e^{c_i}\tilde{\omega}$ and
$\gamma_i^*\tilde{f}=\tilde{f}+c_i$, where $c_0=0$ and $c_1=c$, so the
quotient $M=\tilde{M}/\Gamma$ is a conformal symplectic manifold with
presentation $(\tilde{M},\tilde{\omega},\Gamma,\chi,\tilde{f})$, where
$\chi\colon\Gamma\to\R$ is defined by $\chi(\gamma_i)=c_i$.  Since
$\ker(\chi)=\inner{\gamma_0}$, the minimal presentation of $M$ is
$\tilde{M}/\inner{\gamma_0}$ and the rank of $M$ is $1$.  As a
conformal symplectic manifold $M$ is the contact mapping torus of the
contact manifold $Q/\inner{h_0}$ with respect to the contact
automorphism $h_1\colon Q/\inner{h_0}\to Q/\inner{h_0}$.  The Lee
vector field of $M$ is $A=p_*\tilde{A}$, which is periodic since
$\exp\tilde{A}=\gamma_0$.  To compute the structure of $M$ as a
manifold we may assume that $\zeta_1=\zeta_2=\cdots=\zeta_n=1$, in
which case $G$ is equal to $\GL_0(n,\R)$, the identity component of
$\GL(n,\R)$, which acts transitively on $Q$.  Using the Mostow
decomposition theorem for homogeneous spaces we get
\begin{align*}
Q&\cong\GL_0(n,\R)/\GL_0(n-1,\R)\cong(\SO(n)\times\R^n)/\SO(n-1)\\
&\cong\R\times(\SO(n)\times\R^{n-1})/\SO(n-1)\cong\R\times
T^*\bb{S}^{n-1},
\end{align*}
and therefore
\[M\cong\bb{S}^1\times\bb{S}^1\times T^*\bb{S}^{n-1}.\]

Next we display some Hamiltonian actions of Lee type on $M$.  A
maximal compact subgroup of $G$ is
\[
G_0=G\cap\SO(n),
\]
which is the identity component of
$\{\,g\in\SO(n)\mid[g,\zeta]=0\,\}$.  Let $K_0$ be any closed
connected subgroup of $G_0$ which centralizes $h_1$ and let
$\tilde{K}$ be the closed connected subgroup $K_0\exp(\R\zeta)$ of
$G_0$.  Then the $\tilde{K}$-action on $\tilde{M}$ descends to a
$\tilde{K}$-action on $M=\tilde{M}/\Gamma$.  The subgroup
$\inner{h_0}$ of $\tilde{K}$ acts trivially on $M$, so we get an
action of the compact Lie group
\[
K=\tilde{K}/\inner{h_0}\cong\exp(\R\zeta)/\exp(\Z\zeta)\times
K_0\cong\bb{S}^1\times K_0
\]
on $M$, which is Hamiltonian and of Lee type with central Lee element
$\zeta$.  The standard action of $\SO(n)$ on $\R^{2n}$ is Hamiltonian
relative to $\tilde{\omega}$ with moment map given by
$\tilde{\Phi}(x,y)=xy^\top-yx^\top$.  The contact manifold $Q$ is not
compact and the moment map for the $K_0$-action on $Q$ is not proper
(unless $K_0=\SO(n)$), but nevertheless the moment body $\Delta_0(Q)$
for the $K_0$-action is convex, in fact a rational convex polyhedral
cone, by Theorem~\ref{theorem;real} and Example~\ref{example;real}
below.  Therefore the moment body for the $K$-action is
$\Delta(Q)=\{1\}\times\Delta_0(Q)\subseteq\R\times\lie{k}_0^*$, and it
follows from Theorem~\ref{theorem;conformal-first} that
$\Delta(M)=\Delta(Q)$.

%%%%%%%%%%%%%%%%%%%%%%%%%%%%%%%%%%%%%%%%%%%%%%%%%%%%%%%%%%%%%%%%%%%%%%%%
\subsection*{Contact mapping tori of rank $\ge2$}
%%%%%%%%%%%%%%%%%%%%%%%%%%%%%%%%%%%%%%%%%%%%%%%%%%%%%%%%%%%%%%%%%%%%%%%%

Higher rank contact mapping tori with interesting group actions are
harder to come by.  Below is a slightly contrived class of examples,
which is based on the simple observation that the product $P\times Q$
of an exact symplectic manifold $(P,\omega_P=d\alpha_P)$ and a contact
manifold $(Q,\alpha_Q)$ is a contact manifold with contact form
$\alpha=\pr_P^*\alpha_P+\pr_Q^*\alpha_Q$.

Suppose we are given $k$ commuting strict contactomorphisms $g_i\colon
Q\to Q$.  These generate a strict contact action of $\Gamma=\Z^k$,
which is not necessarily free or proper.  Now let us take $P=T^*\R^k$
with its standard $1$-form $\alpha_0=-\sum_jy_j\,dx_j$ and standard
symplectic structure $\omega_0=d\alpha_0$ and form the contact
manifold $T^*\R^k\times Q$.  Then $\Gamma=\Z^k$ acts freely and
properly on $\R^k$ by translations, and the cotangent lift of this
action to $T^*\R^k$ preserves $\alpha_0$.  Therefore the diagonal
$\Gamma$-action on $T^*\R^k\times Q$ is a proper and free strict
contact action.  This allows us for any choice of real numbers $c_1$,
$c_2,\dots$,~$c_k$ to form the contact mapping torus of $T^*\R^k\times
Q$,
\[
\hM=\hM_{(g_1,g_2,\dots,g_k;c_1,c_2,\dots,c_k)}=(\R\times
T^*\R^k\times Q)/\Gamma,
\]
which we will call the \emph{fattened contact mapping torus} of $Q$
with respect to the $g_i$ and the $c_i$.  Write points of $\R\times
T^*\R^k\times Q$ as quadruplets $(t,x,y,q)$ with $t\in\R$,
$x$,~$y\in\R^k$ and $q\in Q$.  Let $\tilde{f}(t,x,y,q)=t$ and
\[
\tilde\omega=
e^t\bigl(d(\alpha_0+\alpha_Q)+dt\wedge(\alpha_0+\alpha_Q)\bigr).
\]
Also define $\chi\colon\Gamma\to\R$ by $\chi(\gamma_i)=c_i$, where the
$\gamma_i$ are the standard generators of $\Gamma$.  Then $(\R\times
T^*\R^k\times Q,\tilde\omega,\Gamma,\chi,\tilde{f})$ is a presentation
of $\hM$.  If the $c_i$ are independent over $\Q$, this presentation
is minimal and $\hM$ has rank $k$.

The fattened mapping torus $\hM$ relates to the usual mapping torus
$M$ as follows.  The left translation action of $\R^k$ on $\R\times
T^*\R^k\times Q$ defined by $u\cdot(t,x,y,q)=(t,x+u,y,q)$ descends to
a Hamiltonian $\R^k$-action on $\hM$ whose moment map is linear
momentum $\mu([t,x,y,q])=y$.  The reduced space of $\hM$ with respect
to $\R^k$ is
\[\hM\qu\R^k=M=M_{(g_1,g_2,\dots,g_k;c_1,c_2,\dots,c_k)}.\]
The reduced space $M$ exists as a manifold if and only if the
$\Gamma$-action on $\R\times Q$ is free and proper.  If $M$ is not a
manifold, we can interpret it as a conformal symplectic stack in the
spirit of~\cite{HSZ18}: the zero fibre
$\mu^{-1}(0)=(\R\times\R^k\times Q)/\Gamma$ is a conformal
presymplectic manifold in the sense of
Appendix~\ref{section;conformal-presymplectic}, and the quotient map
$\mu^{-1}(0)\to M$ is an atlas of the stack $M$.

Suppose $K$ acts on $Q$ by strict contactomorphisms.  Such an action
is Hamiltonian with moment map $\Psi\colon Q\to\lie{k}^*$ given by
$\Psi^\xi(q)=-\iota(\xi_Q)\alpha_Q$.  Supposing that the $K$-action
commutes with the $\Gamma$-action, we get a Hamiltonian $K$-action on
the fattened contact mapping torus defined by
$k\cdot[t,x,y,q]=[t,x,y,k\cdot q]$.  The moment map
$\widehat{\Phi}\colon\hM\to\lie{k}^*$ is given by
$\widehat{\Phi}([t,x,y,q])=\Psi(q)$ and the moment body
$\Delta(\hM)=\widehat{\Phi}(\hM)\cap C$ is equal to $\Delta(Q)$.

Examples of contact mapping tori of arbitrary rank can now be produced
from the contact quadrics $(Q_\zeta,\alpha)$ considered earlier in
this section.  In the elliptic case, where $K$ is a subgroup of the
centralizer $Z_{\U(n)}(\zeta)$, any choice of elements $t_1$,
$t_2,\dots$,~$t_k$ of the centre $Z(K)$, which is a subgroup of the
maximal torus $T$ of $\U(n)$, generates a strict contact action of
$\Gamma$ on $Q_\zeta$ which commutes with $K$.  In the hyperbolic
case, where $K$ is a compact subgroup of
$Z_{\GL(n,\R}(\zeta)/\inner{\exp(\zeta)}$, we may take any elements
$h_1$, $h_2,\dots$,~$h_k$ of the split maximal torus $H$ of
$\GL(n,\R)$ which commute with $K$ to generate a strict contact action
of $\Gamma$ on $Q_\zeta$ which commutes with $K$.  In either case, we
can then form the fattened mapping torus $\hM$, which is a conformal
Hamiltonian $K$-manifold of Lee type with moment body the convex
polyhedron $\Delta(\hM)=\Delta(Q_\zeta)$.

%%%%%%%%%%%%%%%%%%%%%%%%%%%%%%%%%%%%%%%%%%%%%%%%%%%%%%%%%%%%%%%%%%%%%%%%
\appendix
%%%%%%%%%%%%%%%%%%%%%%%%%%%%%%%%%%%%%%%%%%%%%%%%%%%%%%%%%%%%%%%%%%%%%%%%

%%%%%%%%%%%%%%%%%%%%%%%%%%%%%%%%%%%%%%%%%%%%%%%%%%%%%%%%%%%%%%%%%%%%%%%%
\section{An instance of real symplectic convexity}\label{section;real}
%%%%%%%%%%%%%%%%%%%%%%%%%%%%%%%%%%%%%%%%%%%%%%%%%%%%%%%%%%%%%%%%%%%%%%%%

The next theorem, which is a special case of a result
from~\cite{OS00}, is used in Section~\ref{section;examples} (contact
hyperboloids).  Let $K$ be a compact connected Lie group with a choice
of maximal torus $T$ and a closed chamber $C$ in $\lie{t}^*$.  The
result concerns certain real algebraic subvarieties of a
finite-dimensional unitary $K$-module $V$.  The imaginary part of the
Hermitian inner product $\inner{{\cdot},{\cdot}}_V$ on $V$ is a
symplectic form $\omega_V$, and a moment map for the $K$-action on $V$
is given by
\[\Phi_V^\xi(v)=\frac12\omega_V(\xi(v),v)=\frac{i}2\inner{v,\xi(v)}_V\]
for $\xi\in\lie{k}$ and $v\in V$.  Let $W=V^\C=\C\otimes V$ be the
complexified $K$-module, let $J$ be the complex structure of $V$, and
let $W=V^{1,0}\oplus V^{0,1}$ be the usual eigenspace decomposition of
$W$ respect to $J$.  We have $V^{0,1}=\overline{V^{1,0}}$, where the
conjugation is defined with respect to the totally real subspace $V$
of $W$.  The map $f(v)=\frac12(v-iJv)$ is a complex linear isomorphism
$V\cong V^{1,0}$.  There is a unique $K$-module structure on $V^{1,0}$
that makes $f$ a $K$-module isomorphism.  Likewise, the map
$\bar{f}(v)=\frac12(v+iJv)$ is a complex antilinear isomorphism
$V\cong V^{0,1}$, and there is a unique $K$-module structure on
$V^{0,1}$ that makes $\bar{f}$ an anti-isomorphism of $K$-modules.
Thus $W=V^{1,0}\oplus V^{0,1}$ becomes a complex $K\times K$-module
isomorphic to $V\oplus V^*$.  We will write elements of $W$ as pairs
$\bigl(f(v_1),\bar{f}(v_2)\bigr)$ with $v_1$, $v_2\in V$.  In this
notation complex conjugation is written as
\[
\bigl(f(v_1),\bar{f}(v_2)\bigr)\longmapsto
\bigl(f(v_2),\bar{f}(v_1)\bigr).
\]
The real inner product $\Re\inner{{\cdot},{\cdot}}_V$ on $V$ extends
to a Hermitian inner product $\inner{{\cdot},{\cdot}}_W$ on $W$, and
the symplectic form $\omega_W=\Im\inner{{\cdot},{\cdot}}_W$ satisfies
$f^*\omega_W=\frac12\omega_V$ and
$\bar{f}^*\omega_W=-\frac12\omega_V$.  Therefore the moment map for
the $K\times K$-action on $W$ is given by
\begin{equation}\label{equation;real}
\Phi_W^{(\xi_1,\xi_2)}\bigl(f(v_1),\bar{f}(v_2)\bigr)=
\frac12\bigl(\Phi_V^{\xi_1}(v_1)-\Phi_V^{\xi_2}(v_2)\bigr).
\end{equation}

\begin{theorem}\label{theorem;real}
Let $V$ be a unitary $K$-module.  Let $X$ be a $K$-invariant
irreducible real algebraic subvariety of $V$.  Let $X^\C\subseteq
V^\C$ be the complexification of $X$.  Suppose that $X^\C$ is $K\times
K$-invariant and that $X$ contains at least one smooth point of
$X^\C$.  Then the moment body $\Delta(X)=\Phi_V(X)\cap C$ of $X$ is
the convex hull of all dominant weights $\lambda\in C$ such that the
$K\times K$-module with highest weight $(\lambda,-\lambda)$ occurs in
the coordinate ring of $X^\C$.  In particular $\Delta(X)$ is a
rational convex polyhedral cone in $\lie{k}^*$.
\end{theorem}

\begin{proof}
The strategy is to interpret $\lie{k}$ as the tangent space to the
symmetric space $(K\times K)/K$ and $X$ as a totally real Lagrangian
in $X^\C$.  Let $U=K\times K$ and $W=V^\C$.  Define involutions
$\sigma=\sigma_U$ of $U$ and $\sigma=\sigma_W$ of $W$ by
$\sigma_U(k_1,k_2)=(k_2,k_1)$ and $\sigma_W(w)=\bar{w}$.  Then
$\sigma_W$ is antilinear and antisymplectic, and
$X=(X^\C)^{\sigma_W}$, the set of $\sigma_W$-fixed points of $X$.
Moreover,
\begin{align*}
\sigma_W\bigl((k_1,k_2)\cdot(f(v_1),\bar{f}(v_2))\bigr)&=
\sigma_W\bigl(f(k_1v_1),\bar{f}(k_2v_2)\bigr)\\
&=\bigl(f(k_2v_2),\bar{f}(k_1v_1)\bigr)\\
&=\bigl(k_2,k_1)\cdot(f(v_2),\bar{f}(v_1)\bigr),
\end{align*}
which shows that $\sigma_W(u\cdot w)=\sigma_U(u)\cdot\sigma_W(w)$ for
all $u\in U$ and $w\in W$.  It follows that
$\Phi_W(\sigma_W(w))=-\sigma_U^*\bigl(\Phi_W(w)\bigr)$ for $w\in W$,
where $\sigma_U^*$ is the transpose of the involution $\sigma_{U,*}$
of $\lie{u}=\Lie(U)$ induced by $\sigma_U$.  Let
$\lie{p}^*=\{\,\mu\in\lie{u}^*\mid \sigma_U^*(\mu)=-\mu\,\}$.  Then
\[
\lie{p}^*=
\{\,(\lambda,-\lambda)\mid\lambda\in\lie{k}^*\,\}\cong\lie{k}^*,
\]
and the image $\Phi_W(X)$ is a subset of $\lie{p}^*$.  It follows
from~\eqref{equation;real} that
\[
\Phi_W^{(\xi,-\xi)}\bigl(f(v),\bar{f}(v)\bigr)=
\frac12\bigl(\Phi_V^{\xi}(v)-\Phi_V^{-\xi}(v)\bigr)=\Phi_V^{\xi}(v),
\]
so $\Phi_V=\Phi_W|_V$ modulo the identification
$\lie{p}^*\cong\lie{k}^*$.  Hence $\Phi_V(X)=\Phi_W(X)$.  The theorem
now follows from~\cite[Theorem~6.3]{OS00}.
\end{proof}

\begin{example}\label{example;real}
Let $V=T^*\R^n=\R^{2n}$ equipped with the standard symplectic
structure $\omega$ and complex structure
$J=\bigl(\begin{smallmatrix}0&-I\\I&0\end{smallmatrix}\bigr)$.
Then $V$ is a unitary $K$-module for any closed subgroup $K$ of
$\group{O}(n,\R)$.  Let $\zeta=\diag(\zeta_1,\zeta_2,\dots,\zeta_n)$,
where $\zeta_1\ge\zeta_2\ge\cdots\ge\zeta_n>0$ and suppose that $K$
centralizes $\zeta$.  Then the action of $K$ preserves the hyperboloid
\[X=\{\,(x,y)\in V\mid x^\top\zeta y=1\,\}.\]
To complexify the hyperboloid we substitute new variables
$x=\frac12(z+\bar{z})$ and $y=\frac1{2i}(z-\bar{z})$ and obtain
\[
X^\C=\{\,(z,\bar{z})\in V^{1,0}\oplus V^{0,1}\mid z^\top\zeta
z-\bar{z}^\top\zeta\bar{z}=4i\,\}.
\]
Evidently $X^\C$ is unchanged by substitutions of the form $z\mapsto
k_1z$, $\bar{z}\mapsto k_2\bar{z}$ with $k_1$, $k_2\in K$.  The
variables $z$ are coordinates on $V^{1,0}$ and the variables $\bar{z}$
are coordinates on $V^{0,1}$, so we see that $X^\C$ is preserved by
the action of $K\times K$.  It now follows from
Theorem~\ref{theorem;real} that the moment body $\Delta(X)$ of $X$
with respect to the $K$-action is a rational convex polyhedral cone.
\end{example}

\begin{example}\label{example;unreal}
Let $V=\C^n$ equipped with the standard unitary structure.  Then $V$
is a unitary $K$-module for any closed subgroup $K$ of $\U(n)$.  Let
$\zeta$ be as in Example~\ref{example;real} and suppose that $K$
centralizes $\zeta$.  Then $K$ acts on the ellipsoid
\[X=\{\,z\in V\mid z^*\zeta z=1\,\}.\]
The complexified ellipsoid is
\[
X^\C=\{\,(z,\bar{z})\in V^{1,0}\oplus V^{0,1}\mid\bar{z}^\top\zeta
z=1\,\},
\]
which is not preserved by the action of $K\times K$, so
Theorem~\ref{theorem;real} does not apply.  Obviously the moment body
$\Delta(X)$ cannot be a cone because $X$ is compact.  However, the
contact manifold $X$ is the level set of the $\zeta$-component of the
$\U(n)$-action on $\C^n$, so it follows from the contact convexity
theorem~\cite[Theorem~4.5.1]{LS17} that $\Delta(X)$ is a (not
necessarily rational) convex polytope.
\end{example}

%%%%%%%%%%%%%%%%%%%%%%%%%%%%%%%%%%%%%%%%%%%%%%%%%%%%%%%%%%%%%%%%%%%%%%%%
\section{Conformal presymplectic convexity}
\label{section;conformal-presymplectic}
%%%%%%%%%%%%%%%%%%%%%%%%%%%%%%%%%%%%%%%%%%%%%%%%%%%%%%%%%%%%%%%%%%%%%%%%

The conformal symplectic convexity theorem,
Theorem~\ref{theorem;main}, extends without difficulty to conformal
presymplectic manifolds, where the $2$-form is of constant rank less
than the dimension.  The purpose of this appendix is to briefly
explain this generalization, which uses the techniques of~\cite{LS17}.

%%%%%%%%%%%%%%%%%%%%%%%%%%%%%%%%%%%%%%%%%%%%%%%%%%%%%%%%%%%%%%%%%%%%%%%%
\subsection*{Transverse vector fields and basic differential forms}
%%%%%%%%%%%%%%%%%%%%%%%%%%%%%%%%%%%%%%%%%%%%%%%%%%%%%%%%%%%%%%%%%%%%%%%%

Let $\F$ be a foliation of a manifold $P$ and $\X(\F)$ the Lie algebra
of vector fields tangent to $\F$.  Recall from
Section~\ref{section;conformal-pre} that the Lie algebra of foliate
vector fields is the normalizer $N_{\X(P)}(\X(\F))$ of $\X(\F)$ inside
the Lie algebra $\X(P)$ of all vector fields.  The quotient Lie
algebra
\[\X(P,\F)=N_{\X(P)}(\X(\F))/\X(\F)\]
is the Lie algebra of \emph{transverse vector fields} of the
foliation.  A transverse vector field $X$ is not a vector field, but
an equivalence class of vector fields.  The flow of $X$ maps leaves of
$\F$ to leaves of $\F$, but is well-defined only up to a flow along
the leaves of $\F$.  On every transversal of the foliation $X$ induces
a genuine vector field.  (Hence $\X(M,\F)=\X(M/\F)$ if the leaf space
$M/\F$ is a manifold.)

A differential form $\alpha\in\Omega^*(M)$ is \emph{$\F$-basic} if
$\iota(X)\alpha=\iota(X)d\alpha=0$ for all $X\in\X(\F)$.  We denote
the complex of $\F$-basic forms by $\Omega^*(M,\F)$ and the algebra of
$\F$-basic functions by $\ca{C}^\infty(M,\F)=\Omega^0(M,\F)$.

%%%%%%%%%%%%%%%%%%%%%%%%%%%%%%%%%%%%%%%%%%%%%%%%%%%%%%%%%%%%%%%%%%%%%%%%
\subsection*{Conformal presymplectic forms}
%%%%%%%%%%%%%%%%%%%%%%%%%%%%%%%%%%%%%%%%%%%%%%%%%%%%%%%%%%%%%%%%%%%%%%%%

A \emph{conformal presymplectic structure} on a manifold $M$ is a pair
$(\omega,\theta)$, where $\omega$ is a $2$-form of constant rank and
$\theta$ a closed $1$-form satisfying $d_\theta\omega=0$ and
$\iota(X)\theta=0$ for all vector fields $X$ with $\iota(X)\omega=0$.
As in the nondegenerate case we call $\theta$ the \emph{Lee form} and
its cohomology class the \emph{Lee class}, and we say
$(\omega,\theta)$ is \emph{strict} if $\theta$ is not exact, and
\emph{global} if $\theta$ is exact.

Let $(M,\omega,\theta)$ be a conformal presymplectic manifold.  The
kernel of $\omega$ is an involutive subbundle of $TM$, which generates
a (regular) foliation $\F=\F_\omega$ called the \emph{null foliation}
of $\omega$.  The forms $\omega$ and $\theta$ are basic with respect
to this foliation, and $(\omega,\theta)$ induces a conformal
symplectic structure on every transversal of the foliation.  In this
sense a conformal presymplectic structure is a transverse analogue of
a conformal symplectic structure.

Suppose a function $f\in\ca{C}^\infty(M)$ and a vector field
$X\in\X(M)$ satisfy $d_\theta f=\iota(X)\omega$.  Then $f$ must be
$\F$-basic, and $X$ is determined by $f$ only up to an element of
$\X(\F)$.  Moreover, $L_\theta(X)\omega=0$, so
\[\omega([X,Y],Z)=-\omega(Y,[X,Z])-\theta(X)\omega(Y,Z)\]
for all $Y$, $Z\in\X(M)$.  This shows that $[X,Y]\in\X(\F)$ for all
$Y\in\X(\F)$; in other words $X$ is foliate and therefore represents a
transverse vector field $X_f\in\X(M,\F)$, called the \emph{Hamiltonian
  vector field} of $f$.  Thus the Hamiltonian correspondence $f\mapsto
X_f$ is a map
\[\ca{C}^\infty(M,\F)\longto\X(M,\F).\]
In particular the \emph{Lee vector field} is the transverse vector
field $A\in\X(M,\F)$ defined by $A=X_1$.

Suppose $K$ acts on $M$ by foliate transformations.  Then the
infinitesimal action $\lie{k}\to\X(M)$ takes values in
$N_{\X(\F)}(\X(M))$.  Composing this map with the quotient map gives a
homomorphism $\lie{k}\to\X(M,\F)$.  As in the nondegenerate case, we
say the action is \emph{weakly Hamiltonian} if there is a linear
lifting $\Phi^\vee$ of this homomorphism,
\[
\begin{tikzcd}
&\ca{C}^\infty(M,\F)\ar[d]\\
\lie{k}\ar[ur,pos=0.6,dotted,"\Phi^\vee"]\ar[r]&\X(M,\F),
\end{tikzcd}
\]
and the map $\Phi\colon M\to\lie{k^*}$ dual to $\Phi^\vee$ is the
\emph{moment map}.  The $K$-action is \emph{Hamiltonian}, or $M$ is a
\emph{Hamiltonian conformal presymplectic $K$-manifold}, if there
exists a moment map $\Phi$ which is $K$-equivariant.  The action is of
\emph{Lee type} if the transverse Lee vector field $A$ has a foliate
representative of the form $\zeta_M$ with $\zeta\in\lie{k}$.
We can now state the conformal presymplectic convexity theorem.

\begin{theorem}\label{theorem;conformal-presymplectic}
Let $K$ be a compact connected Lie group which acts on a connected
strict conformal presymplectic manifold $(M,\omega,\theta)$ in a
Hamiltonian fashion.  Assume that the $K$-action is clean and of Lee
type, and that the moment map $\Phi\colon M\to\lie{k}^*$ is proper.
Choose a maximal torus $T$ of $K$ and a closed Weyl chamber $C$ in
$\lie{t}^*$, where $\lie{t}=\Lie(T)$.
\begin{enumerate}
\item
The fibres of $\Phi$ are connected and $\Phi\colon M\to\Phi(M)$ is an
open map.
\item
$\Delta(M)=\Phi(M)\cap C$ is a closed convex polyhedral set.
\end{enumerate}
\end{theorem}

\begin{proof}[Outline of proof]
The essential fact that $d_\theta\colon\ca{C}^\infty(M)\to\Omega^1(M)$
is injective (Proposition~\ref{proposition;strict}) remains valid for
conformal presymplectic manifolds, as is clear from the proof
in~\cite{Va85}.  None of the auxiliary results in
Section~\ref{section;hamilton} uses the nondegeneracy of~$\omega$.
The tubular neighbourhood theorem,
Proposition~\ref{proposition;local-model}, generalizes as follows: for
every $m\in M$ and every equivariant tubular neighborhood $V$ of
$K\cdot m$ there exists a $K$-invariant smooth function $f$ on $V$
such that $\theta|_V=df$; the form $\Omega=e^f\cdot\omega|_V$ is
presymplectic, the $K$-action on $(V,\Omega)$ is Hamiltonian with
equivariant moment map $\Psi=e^f\cdot\Phi|_V$, and the null foliation
of $\Omega$ is $\F_\Omega=\F_\omega|_V$.  In particular, the
$K$-action on $V$ is clean with respect to the foliation $\F_\Omega$.
In the proof of the local convexity theorem,
Theorem~\ref{theorem;local-convexity}, instead of appealing to the
symplectic version of the local convexity theorem, we need to appeal
to the presymplectic local convexity
theorem,~\cite[Theorem~2.12.1]{LS17}.  The remainder of the proof, and
the proof of Theorem~\ref{theorem;main}, then go through unchanged.
\end{proof}

%%%%%%%%%%%%%%%%%%%%%%%%%%%%%%%%%%%%%%%%%%%%%%%%%%%%%%%%%%%%%%%%%%%%%%%%

\bibliographystyle{amsplain}

\bibliography{conformal}

%%%%%%%%%%%%%%%%%%%%%%%%%%%%%%%%%%%%%%%%%%%%%%%%%%%%%%%%%%%%%%%%%%%%%%%%
\end{document}